\documentclass{amsart}

\usepackage{mathtools, amssymb}
\usepackage{bbm}
\usepackage{doi}
\usepackage{yfonts}
\usepackage[utf8]{inputenc}

\DeclareMathSizes{10}{10}{8}{7}
\DeclareMathSizes{11}{11}{9}{8}
\DeclareMathSizes{12}{12}{10}{9}

\newtheorem{theorem}{Theorem}[section]
\newtheorem{lemma}[theorem]{Lemma}
\newtheorem{proposition}[theorem]{Proposition}
\newtheorem{corollary}[theorem]{Corollary}
\newtheorem{paradox}[theorem]{Paradox}

\theoremstyle{definition}
\newtheorem{definition}[theorem]{Definition}
\newtheorem{question}[theorem]{Question}

\theoremstyle{remark}
\newtheorem{remark}[theorem]{Remark}
\newtheorem{claim*}[theorem]{Claim}

\newcommand{\dom}{\mathrm{dom}}

\newcommand{\otp}{\mathrm{otp}}

\newcommand{\otQ}{\eta}

\newcommand{\soast}[2]{\mathord \{{#1} : {#2}\}}
\newcommand{\bigsoast}[2]{\mathord \big\{#1 \bigm\vert #2\big\}}
\newcommand{\Bigsoast}[2]{\mathord \Big\{#1 \Bigm\vert #2\Big\}}
\newcommand{\soafft}[2]{{}^{#1} #2}

\newcommand{\Bigsetminus}{\mathbin{\Big\backslash}}
\newcommand{\seq}[2]{\langle {#1} : {#2}\rangle}

\newcommand{\cf}{cf.}
\newcommand{\Cf}{Cf.}
\newcommand{\ie}{i.e.}
\newcommand{\card}[1]{\mathord | #1 |}
\newcommand{\pwim}[2]{
#1\big[#2\big]
}
\newcommand{\ifff}{if and only if}
\newcommand{\dfeq}{:=}
\newcommand{\eqdf}{=:}
\newcommand{\wlg}{w.l.o.g.}
\newcommand{\QQ}{\mathbb{Q}}
\newcommand{\ltl}{<_\mathrm{lex}}

\newcommand{\disj}{\mathrel{\scriptstyle\vee}}
\newcommand{\Pot}{\mathcal{POW}}
\newcommand{\lnth}{\ell}
\newcommand{\doesntarrow}{{\mspace{7mu}\not\mspace{-7mu}\longrightarrow}}
\newcommand{\arrows}{\longrightarrow}
\newcommand{\stick}{{\ensuremath \mspace{2mu}\mid\mspace{-12mu} {\raise0.6em\hbox{$\bullet$}}}}
\newcommand{\stickp}{
\stick
}

\newcommand{\opair}[2]{
\mathord \langle #1, #2\rangle
}
\newcommand{\opin}[3]{
\mathord {{\mathopen ]} {#1} {\mathpunct ,} \ {#2} {\mathclose [}}_{#3}
}
\newcommand{\loin}[2]{
\mathord {{\mathopen ]} {#1} {\mathpunct , } \ {#2} {\mathclose ]}}
}
\newcommand{\roin}[2]{
\mathord {{\mathopen [} {#1} {\mathpunct ,} \ {#2} {\mathclose [}}
}

\newcommand{\iz}{is}
\newcommand{\ou}{ou}

\DeclareMathOperator{\cof}{cof}
\DeclareMathOperator{\cov}{\mathsf{cov}}
\DeclareMathOperator{\add}{\mathsf{add}}

\DeclareMathOperator{\meagre}{\mathcal{M}}
\DeclareMathOperator{\zero}{\mathcal{N}}
\DeclareMathOperator{\MA}{MA}
\DeclareMathOperator{\ZFC}{ZFC}

\begin{document}
\title{Partitioning subsets of generalised scattered orders}
\author{Chris Lambie-Hanson}
\address{Department of Mathematics, Bar-Ilan University\\
  Ramat Gan, 5290002\\
  Isra\"{e}l}
\email{lambiec@macs.biu.ac.il}
\thanks{A portion of this research was undertaken while the first author was a Lady Davis Postdoctoral Fellow and the second author was a postdoctoral fellow at the Ben-Gurion-University of the Negev. The first author would like to thank the Lady Davis Fellowship Trust and the Hebrew University of Jerusalem, as well as Bar-Ilan University and the Israel Science Foundation (grant \#1630/14) for supporting this research. The second author would like to thank the Ben-Gurion University of the Negev and the Israel Science Foundation which supported this research(grant \#1365/14). Both authors wish to thank William Chen, Fred Galvin, Menachem Kojman and Ashutosh Kumar for discussing  proofs, results, or problems of this paper. They also wish to thank the logic and topology seminar of Ben-Gurion University organised by Nadav Meir and Omer Mermelstein for hosting a pair of lectures which triggered this cooperation. Finally they would like to thank Heike Mildenberger who shared her \LaTeX-Code for $\stick$ with them.
}
\author{Thilo Weinert}
\address{Universit\"at Wien\\
Kurt G\"odel Research Centre for Mathematical Logic\\
W\"ahringer Straße 25\\
1090 Wien\\
Autriche}
\email{thilo.weinert@univie.ac.at}


\begin{abstract}
  In 1956, 48 years after Hausdorff provided a comprehensive account on ordered sets and defined
  the notion of a scattered order, Erd\H{o}s and Rado  founded the partition calculus in a seminal paper.
    The present paper gives an account of investigations into
  general\iz ations of scattered linear orders and their partition relations for both singletons and pairs. We consider analogues for these order-types of  known partition theorems for ordinals or scattered orders and prove a partition theorem from assumptions about cardinal characteristics.
  Together, this continues older research by Erd\H{o}s, Galvin, Hajnal, Larson and Takahashi and more recent investigations by
  Abraham, Bonnet, Cummings, D\v{z}amonja, Komj\'{a}th, Shelah and Thompson.
\end{abstract}
\maketitle

\section{Introduction}

In this paper, we study partition relations in the context of scattered linear orders and
their general\iz ations. Recall that a linear order is \emph{scattered} if it does not embed
a copy of the rationals and is \emph{$\sigma$-scattered} if it is a countable
union of scattered linear orders. The classes of scattered and $\sigma$-scattered linear orders
are, in a sense, the simplest classes of linear orders past the class of well-orders, and, as such,
they have played a central role in the study of general linear orders.

One of the central questions motivating the investigations of this paper concerns the extent to which
the classes of scattered and $\sigma$-scattered linear orders behave similarly to the class of well-orders.
Two seminal results contributing to the understanding of this issue are the following:
\begin{itemize}
  \item In \cite{908H0}, Hausdorff character\iz es the class of scattered linear orders as the
    smallest class containing all well-orders and anti-well-orders and closed under well-ordered
    and anti-well-ordered lexicographic sums.
  \item In \cite{971L0}, Laver proves Fra\"iss\'e's Conjecture by showing that the class of $\sigma$-scattered
    linear orders is well-quasi-ordered by embeddability.
\end{itemize}

We will be interested not just in scattered linear orders but also in general\iz ations of scattered linear
orders introduced by D\v{z}amonja and Thompson in \cite{006DT0}. We first recall some relevant definitions.
For unfamiliar notation, see the Notation subsection at the end of the Introduction.

\begin{definition}
  Suppose $\kappa$ is an infinite cardinal, $\varphi$ is a linear order type, and $P$ is a linear order of type $\varphi$.
  \begin{enumerate}
    \item $\varphi$ is \emph{$\kappa$-dense} if $\card{P} > 1$ and, for all $a,b \in P$, if $a <_P b$, then
    $\card{\opin{a}{b}{P}} \geqslant \kappa$.
    \item $\varphi$ is \emph{$\kappa$-saturated} if $P \not= \emptyset$ and, for all $A,B \subseteq P$ such that $\card{A}, \card{B} < \kappa$
      and $A <_P B$, there is $c \in P$ such that $A <_P c <_P B$.
  \end{enumerate}
\end{definition}

In \cite{908H0, 914H0} Hausdorff shows that, if $\kappa$ is an infinite cardinal, the following hold:
\begin{itemize}
  \item If $P$ is a $\kappa$-saturated linear order, then every linear order of cardinality $\kappa$ embeds into $P$.
  \item There is a $\kappa^+$-saturated linear order $L$ of cardinality $2^\kappa$ such that:
  \begin{itemize}
  \item $L$ is embeddable into every $\kappa^{+}$-saturated linear order;
  \item  no suborder of $L$ of cardinality $\kappa^{++}$ is well-ordered or anti-well-ordered.
  \end{itemize}
  \item If $\kappa$ is singular, then every $\kappa$-saturated order is also $\kappa^+$-saturated.
\end{itemize}

In \cite{949S0}, Sierpi\'{n}ski defines, for every infinite cardinal $\kappa$, a linear order of size $\kappa^{<\kappa}$,
which he calls $\QQ_{\kappa}$. In addition, he shows that, for every $\kappa$, $\QQ_{\kappa^+}$ is $\kappa^+$-saturated.
In \cite{956G0}, Gillman shows:
\begin{enumerate}
  \item for every infinite cardinal $\kappa$, $\QQ_\kappa$ is embeddable into every $\kappa$-saturated linear order;
  \item for every limit cardinal $\kappa$, both that every $\kappa$-sized linear order embeds into $\QQ_\kappa$ and that it is $\kappa$-saturated \ifff\
    $\kappa$ is inaccessible.
\end{enumerate}

In \cite{006DT0} D\v{z}amonja and Thompson, building on previous work of Abraham and Bonnet, introduce the following definition.

\begin{definition}[D\v{z}amonja-Thompson, \cite{006DT0}] \label{gen_scattered_def}
  Suppose $\kappa$ is an infinite, regular cardinal, and $\varphi$ is a linear order type.
  \begin{enumerate}
    \item $\varphi$ is \emph{$\kappa$-scattered} if there is no $\kappa$-dense order type $\tau$ such that $\tau \leqslant \varphi$.
    \item $\varphi$ is \emph{weakly $\kappa$-scattered} if there is no $\kappa$-saturated $\tau$ such that $\tau \leqslant \varphi$.
  \end{enumerate}
\end{definition}

Note that, for $\kappa = \aleph_0$, the classes of $\kappa$-scattered and weakly $\kappa$-scattered linear orders
coincide and are equal to the class of scattered linear orders. For uncountable $\kappa$, the classes are provably
different, i.e., there is a weakly $\kappa$-scattered linear order that is not $\kappa$-scattered.

In \cite{012ABCDT0}, Abraham et al.\ prove a general\iz ation of Hausdorff's structure theorem for the class
of $\kappa$-scattered linear orders.

\begin{definition}
  Suppose $\kappa$ is an infinite cardinal. Then $\mathcal{BL}_\kappa$ denotes the class of all linear
  order types $\varphi$ such that either $\card{\varphi} < \kappa$, $\varphi$ is a well-ordering, or $\varphi$ is
  an anti-well-ordering.
\end{definition}

\begin{theorem}[{\cite[Theorem 3.10]{012ABCDT0}}] \label{hausdorff_structure_thm}
  Let $\kappa$ be an infinite cardinal. The class of $\kappa$-scattered linear order types is
  the smallest class of linear order types containing $\mathcal{BL}_\kappa$ and closed under lexicographic
  sums with index set in $\mathcal{BL}_\kappa$.
\end{theorem}

The results in this paper concern the behavior of general\iz ed scattered orders with respect to partition relations.
The following folklore result provides some motivation for our investigations.
The proof uses a pair-col\ou ring first deployed by Sierpi\'{n}ski, \cf\ \cite{933S0}.

\begin{theorem}[Folklore] \label{sierpinski_thm}
  If $\tau$ is a linear order type, then $\tau \doesntarrow (\omega, \omega^*)^2$.
\end{theorem}

Theorem \ref{sierpinski_thm} implies that, in any consistent positive partition relation
of the form $\tau \arrows (\varphi_0, \varphi_1)^2$, where all variables indicate linear order types,
if $\varphi_0$ and $\varphi_1$ are infinite, then either both must be well-orders or both must be anti-well-orders.
Since we want to work with classes of linear orders larger than the class of well-orders, this
indicates that we should look at partition relations of the form $\tau \arrows (\varphi, n)^2$,
where $n$ is a natural number. Note that, if $\kappa$ is an infinite cardinal and $\alpha < \kappa^+$,
then there is $\beta < \kappa^+$ such that, for all $n < \omega$, we have $\beta \arrows (\alpha, n)^2$, \cf\ \cite{972EM0}. Our goal of comparing the classes of general\iz ed scattered orders with the class of well-orders
thus leads us naturally to the following general question.

\begin{question}
  Suppose $\kappa$ is an infinite cardinal and $\varphi$ is a $\kappa$-scattered (resp.\ weakly $\kappa$-scattered)
  linear order type of size $\kappa$. Must there be a $\kappa$-scattered (resp.\ weakly $\kappa$-scattered) $\tau$
  of size $\kappa$ such that, for all $n < \omega$, $\tau \arrows (\varphi, n)^2$?
\end{question}

In most of the existing literature concerning partition relations and linear orders, the linear orders under
consideration are in fact well-orders. Let us mention the exceptions we are aware of: Erd\H{o}s and Rado
prove in \cite[Theorem 6]{956ER0} that $\otQ \arrows (\otQ, \aleph_0)^2$. Erd\H{o}s and Hajnal
prove in \cite[Corollary 1]{962EH0} that for any countable scattered linear order type $\tau$ we have
that $\tau \arrows (\varphi, \aleph_0)^2$ implies $\varphi \in \soast{n, n + \omega^*,
\omega + n}{n < \omega}$. Larson proves in \cite[Theorem 4.1]{973L0} that
${(\omega\omega^*)}^\omega \arrows \big({(\omega\omega^*)}^\omega, n\big)^2$ for all natural
numbers $n$; she also shows in \cite[Theorem 4.3]{973L0} that ${(\omega\omega^*)}^{kn} \arrows
\big((\omega\omega^*)^k, n\big)^2$ for all natural numbers $k$ and $n$.
\begin{remark}
\label{remark : (alpha*alpha^*)^omega}
Recall that ${(\alpha\alpha^*)}^\omega$ can be defined as the order type of the set $S_\alpha \dfeq \alpha^{<\omega}$ of finite sequences of elements of $\alpha$, ordered by $\prec_\alpha$. In order to define $\prec_\alpha$ let, for $s, t \in S_\alpha$, $\delta(s, t) \dfeq \min\soast{n < \omega}{n \in \dom(s) \triangle \dom(t) \vee s(n) \ne t(n)}$.

For $s, t \in S_\alpha$ let $s \prec_\alpha t \Leftrightarrow \begin{cases}
\delta(s, t) \text{ is even and } (t \sqsubset s \text{ or } t(\delta(s, t)) < s(\delta(s, t))) \text{ or}\\
\delta(s, t) \text{ is odd and } (s \sqsubset t \text{ or } s(\delta(s, t)) < t(\delta(s, t))).
\end{cases}$
\end{remark}
Regarding this order type, \cf\ Corollaries \ref{corollary : milner_rado_generalization} and \ref{corollary : second omega_1_negative_scattered}
where it appears as a factor of some order types in negative partition relations.
Also \cf\ Questions \ref{strong_weakly_scattered_qstn} and \ref{strong_scattered_qstn}.

\smallskip

The paper is structured as follows. In Section \ref{def_section}, we prove some basic results concerning
general\iz ed scattered orders that will be used throughout the paper. 
Furthermore, we show that, if $\tau$ is a $\kappa$-saturated linear order type and $\varphi$ is a linear
order type of size $\kappa$, then $\tau \arrows (\varphi, n)^2$ for all $n < \omega$. In Section
\ref{weakly_scattered_section}, we show that, if $\kappa^{<\kappa} = \kappa$ and $\varphi$ is a weakly
$\kappa$-scattered linear order type of size
$\kappa$, then there is a weakly $\kappa$-scattered linear order type $\tau$ of size $\kappa$ such that
$\tau \arrows (\varphi, n)^2$ for all $n < \omega$. \footnote{For $\kappa = \aleph_0$ this result is probably due to Galvin but has never been published, \cf\ \cite{014G0, 015S1}.
Its noticable absence from \cite{982R0} led the second author to its rediscovery. Subsequently the first author general\iz ed it to its current form.}

In Section \ref{cardinal_characteristics}, we prove that, for a regular cardinal $\kappa$, the negative partition relation
$\kappa^+\kappa \doesntarrow (\kappa^+\kappa, 3)^2$ follows from $\mathfrak{b}_\kappa = \stick(\kappa) = \kappa^+$.
This complements a result of Larson indicating that the same negative relation follows from $\mathfrak{d} = \aleph_1$.
In Section \ref{scattered_section}, we use this negative partition relation to show that, if $\mathfrak{b} = \stick = \aleph_1$
or $\mathfrak{d} = \aleph_1$, then there is an $\aleph_1$-scattered linear order type $\varphi$ of size $\aleph_1$ such that,
for every $\aleph_1$-scattered linear order type $\tau$ of size $\aleph_1$, $\tau \doesntarrow (\varphi, 3)^2$.
We show, moreover, that if $\kappa > \omega$ is a regular cardinal, it is consistent that $\kappa^{<\kappa} = \kappa$ and that
there is a $\kappa$-scattered order type $\varphi$ of size $\kappa$ such that, for all $\kappa$-scattered
order types $\tau$ of size $\kappa$, $\tau \doesntarrow (\varphi, 3)^2$. Finally, in Section
\ref{ks_section}, we general\iz e a result of Komjath and Shelah from \cite{003KS0}. We prove that,
if $\kappa^{<\kappa} = \kappa$, $\varphi$ is a $\kappa$-scattered linear order type, and $\nu$ is a cardinal,
then there is a $\kappa$-scattered linear order type $\tau$ such that $\tau \arrows [\varphi]^1_{\nu, \kappa}$.

\subsection{Notation}
We use Greek minuscules $\rho, \tau, \varphi$ and $\psi$ to refer to order types of linear orders.
In particular, $\omega$ refers to the order type of the natural numbers and $\otQ$ to that of the
rational numbers. The Greek letters $\kappa, \mu$ and $\nu$ refer to cardinals. All other Greek letters
used refer to ordinals. The only exception is $\sigma$ which appears in the notion of
$\sigma$-scatteredness.
We will use Roman capitals to refer to actual ordered sets, $\opair{P}{<_P}$, for instance.
If $A, B \subseteq P$ (with one or both possibly empty),
then $A <_P B$ means that,
for all $a \in A$ and $b \in B$, $a <_P b$. If $A \subseteq P$ and $b \in P$, then
$A <_P b$ and $b <_P A$ have the obvious meanings. If $P$ is a linear order of order type
$\varphi$, we will denote this by $\mathrm{otp}(P) = \varphi$. We will sometimes abuse notation and
write $\card{\varphi}$ to denote $\card{P}$, where $P$ is an order of type $\varphi$. If $\varphi$ and $\tau$ are linear order types, we will
write $\varphi \leqslant \tau$ to mean that there are linear orders $P$ and $Q$ of type $\varphi$ and $\tau$,
respectively, such that $P$ is a suborder of $Q$.

If $a, b \in P$ and $a <_P b$, then $\opin{a}{b}{P}$
denotes the open interval $\soast{c \in \varphi}{a <_P c <_P b}$ and $[a,b]_P$ denotes the closed
interval $\soast{c \in P}{a \leqslant_\varphi c \leqslant_\varphi b}$. $\roin{a}{b}_P$ and $\loin{a}{b}_P$ are given
the obvious meanings. If $s$ is a sequence, then $\lnth(s)$ denotes its length. If $\varphi$ is an order type, then $\varphi^*$ denotes the reverse of $\varphi$. If
$\varphi$ is a linear order type and, for all $a \in \varphi$, $\tau_a$ is a linear order type, then the lexicographic sum $\sum_{a \in \varphi}
\tau_a$ is the type of the linear order consisting of pairs in $X \dfeq \soast{\opair{a}{b}}{a \in P \wedge b \in T_a}$ where $\opair{P}{<_P}$
is an order of type $\varphi$ and, for every $a \in P$, the pair $\opair{T_a}{<_{T_a}}$ is an order of type $\tau_a$.
If $\opair{a_0}{b_0}, \opair{a_1}{b_1} \in X$ are of this kind, then $\opair{a_0}{b_0} <_X \opair{a_1}{b_1}$ iff
$a_0 <_P a_1$ or $(a_0 = a_1$ and $b_0 <_{ T_{a_0}} b_1)$. If there is a linear order type $\tau$ such that $\tau_a = \tau$ for all $a \in \varphi$, then
we may denote the lexicographic sum $\sum_{a \in \varphi} \tau$ as $\tau\varphi$. 

If $\beta$ is an ordinal and $f \neq g$ are functions with $\dom(f) = \dom(g) = \beta$,
then $\Delta(f,g)$ is the least $\alpha < \beta$ such that $f(\alpha) \neq g(\alpha)$. For every sequence of sets $\seq{X_\alpha}{\alpha < \beta}$ we denote the family of all functions having domain $\beta$ which, for all $\alpha < \beta$, satisfy $f(\alpha) \in X_\alpha$, by $\bigtimes_{\alpha < \beta} X_\alpha$.
If $\beta$ is an ordinal and, for all $\alpha < \beta$, $\varphi_\alpha$ is an order type,
then $\prod_{\alpha < \beta} \varphi_\alpha$ is given by the lexicographic ordering on $\bigtimes_{\alpha < \beta} P_\alpha$ where $P_\alpha$ has type $\varphi_\alpha$ for every $\alpha < \beta$. More precisely,
if $f,g \in \bigtimes_{\alpha < \beta} P_\alpha$, then $f \ltl g$ iff
$f(\Delta(f,g)) <_{P_{\Delta(f,g)}} g(\Delta(f,g))$. If there is $\varphi$ such that
$\varphi_\alpha = \varphi$ for all $\alpha < \beta$, we will often write $\soafft{\beta}{\varphi}$
instead of $\prod_{\alpha < \beta} \varphi_\alpha$.
Finally, for every cardinal $\kappa$, let $\QQ_\kappa$ denote the lexicographic ordering of the set of sequences $s$ of zeros and ones of
length $\kappa$ for which there is a largest $\alpha < \kappa$ such that $s(\alpha) = 1$. \Cf\ \cite{913S0, 005H0} for discussions of order types and their arithmetic.

For an ordered set $X$ and an order type $\tau$ the expression $[X]^\tau$ denotes the set of all subsets of $X$ having order type $\tau$. If $Y$ is any set and $\kappa$ is a cardinal then $[X]^\kappa$ denotes the collection of all subsets of $X$ having cardinality $\kappa$.

If $\varphi$ and $\tau$ are linear order types and $\nu$, $\mu$ are cardinals, then $\tau \arrows (\varphi)^\nu_\mu$ holds
if, whenever $P$ has order type $\tau$ and $f:[P]^\nu \rightarrow \mu$, there is a suborder $Q$ of $P$ such that $Q$ has order type
$\varphi$ and $f \restriction [Q]^\nu$ is constant. If $n$ is a natural number, then $\tau \arrows (\varphi_0, \dots, \varphi_n)^\nu$
holds if, whenever $P$ has order type $\tau$ and
$f: [P]^\nu \rightarrow n+1$, there is a suborder $Q$ of $P$ and some $i \leqslant n$ such that $Q$ has order type
$\varphi_i$ and $f \restriction [Q]^\nu$ is constant with value $i$. If $n$ is a natural number and, for all $i \leqslant n$,
$k_i$ is a natural number, then
\begin{align*}
\tau \arrows (\varphi_{0,0} \disj \dots \disj \varphi_{0,k_0}, \dots, \varphi_{n, 0} \disj \dots \disj \varphi_{n, k_n})^\nu
\end{align*}
holds if, whenever $P$ has order type $\tau$ and $f: [P]^\nu \rightarrow n+1$ there is an $i \leqslant n$, an $m \leqslant k_i$
and a suborder $Q$ of $P$ such that $Q$ has order type $\varphi_{i, m}$ and $f \restriction [Q]^\nu$ is constant with value $i$.
In all cases of interest to us, $\nu$ will be finite.

We will also be discussing square-bracket partition relations, which we recall here.  Suppose $\varphi$ and $\psi$ are linear orders,
$1 \leqslant n < \omega$, and $\mu \leqslant \nu$ are cardinals.
Then $\psi \arrows [\varphi]^n_{\nu, < \mu}$ is the assertion that, whenever $P$ has order type $\psi$ and
$f:[P]^n \rightarrow \nu$, there is $A \in [\nu]^{<\mu}$ and a suborder $Q$ of $P$ such that
$Q$ has order type $\varphi$ and $\pwim{f}{[Q]^n} \subseteq A$. If $\kappa$ is a cardinal,
$\psi \arrows [\varphi]^n_{\nu, < \kappa^+}$ is typically written as $\psi \arrows
[\varphi]^n_{\nu, \kappa}$.

This style of notation for partition problems was first introduced in \cite{956ER0}; both the square-bracket partition relation and the partition relation with alternatives were introduced in \cite{965EHR0}. The latter is was not that widely considered up to now as it is superfluous whenever all the entries are ordinals, which is still the most widespread version. To our knowledge, after \cite{965EHR0, 971EMR0, 017LSW0} this is only the fourth paper in which this relation is considered.

\section{Preliminaries} \label{def_section}

To state our results in their full generality, we will be interested in the following notion, which is
slightly finer than that given by Definition \ref{gen_scattered_def}.

\begin{definition}
  Suppose $\kappa$ is an infinite, regular cardinal, $\mu$ is an infinite cardinal,
  $\varphi$ is a linear order type, and $P$ is an order of type $\varphi$. $\varphi$ is \emph{$\opair{\kappa}{\mu}$-scattered}
  (resp.\ \emph{weakly $\opair{\kappa}{\mu}$-scattered}) if there is
  $\nu < \mu$ and a sequence of suborders
  $\langle P_\zeta \mid \zeta < \nu \rangle$ of $P$ such that $\otp(P_\zeta)$ is $\kappa$-scattered
  (resp.\ weakly $\kappa$-scattered) for all $\zeta < \nu$ and $\bigcup_{\zeta < \nu} P_\zeta = P$.
\end{definition}

We start by investigating the relationships between these different classes of linear order types.
Suppose $\kappa$ is an infinite cardinal.
The character\iz ation of the class of $\kappa$-scattered linear order types given by Theorem \ref{hausdorff_structure_thm}
allows one to prove results about $\kappa$-scattered linear order types by
induction on the ``complexity" of the type. More precisely, let $\mathcal{L}$ be the class of $\kappa$-scattered
linear order types. Define $\mathcal{L}_\alpha$ inductively for ordinals $\alpha$ as follows.
$\mathcal{L}_0 = \mathcal{BL}_\kappa$. If $\beta$ is a limit ordinal, then $\mathcal{L}_\beta = \bigcup_{\alpha < \beta}
\mathcal{L}_\alpha$. If $\mathcal{L}_\alpha$ has been defined, then $\mathcal{L}_{\alpha + 1}$ is the class of all
lexicographic sums $\sum_{a \in \tau} \varphi_a$, where $\tau \in \mathcal{BL}_\kappa$ and, for all $a \in \tau$,
$\varphi_a \in \mathcal{L}_\alpha$. Then $\bigcup_{\alpha \in \mathrm{On}} \mathcal{L}_\alpha = \mathcal{L}$.

We use this method here to relate the classes of $\kappa$-scattered and $\opair{\aleph_0}{\kappa}$-scattered linear order types
when $\kappa$ is a successor cardinal.

\begin{proposition}
  Suppose $\kappa = \nu^+$ and $\varphi$ is a $\kappa$-scattered linear order type. Then
  $\varphi$ is $\opair{\aleph_0}{\kappa}$-scattered.
\end{proposition}

\begin{proof}
  Let $P$ be an order of type $\varphi$. We proceed by induction on the complexity of $\varphi$. First suppose $\varphi \in \mathcal{BL}_\kappa$.
  If $\card{\varphi} < \kappa$, then $P$ is the union of fewer than $\kappa$ singletons, so $\varphi$
  is certainly $\opair{\aleph_0}{\kappa}$-scattered. If, on the other hand, $\varphi$ is a well-order or
  anti-well-order, then $\varphi$ is itself scattered and hence $\opair{\aleph_0}{\kappa}$-scattered.

  Next, suppose $\tau \in \mathcal{BL}_\kappa$, $\varphi = \sum_{a \in \tau} \varphi_a$, and each
  $\varphi_a$ is $\opair{\aleph_0}{\kappa}$-scattered. Let $T$ be an order of type $\tau$. For each $a \in T$, let $P_a$ be an order of type
  $\varphi_a$ and let $\soast{P_{a, \zeta}}{\zeta < \nu}$ be a family of suborders of $P_a$ witnessing
  that $\varphi_a$ is $\opair{\aleph_0}{\kappa}$-scattered. We may assume that $P$ consists of the
  set $\soast{\opair{a}{b}}{a \in T \wedge b \in P_a}$ ordered lexicographically.

  Suppose first that $\card{T} < \kappa$. For $a \in T$ and $\zeta < \nu$, let
  $P'_{a, \zeta} = \{a\} \times P_{a, \zeta}$. Then $P'_{a, \zeta}$ is scattered as a suborder of $P$, so
  $P = \bigcup_{a \in T, \zeta < \nu} P'_{a, \zeta}$ is the union of fewer than $\kappa$ scattered linear orders, so $\varphi$ is
  $\opair{\aleph_0}{\kappa}$-scattered.

  Finally, suppose that $T$ is a well-order or anti-well-order.
  Then, for each $\zeta < \nu$, $Q_\zeta \dfeq \bigcup_{a \in T} P'_{a, \zeta}$ is
  scattered as a suborder of $P$, so $P = \bigcup_{\zeta < \nu} Q_\zeta$ and $\varphi$ is
  $\opair{\aleph_0}{\kappa}$-scattered.
\end{proof}

For strongly inaccessible $\kappa$, the preceding proposition fails. To show this, we need some useful
facts about saturated linear orders. Let $\mu$ be an infinite, regular cardinal. An easy argument
shows that there is a $\mu$-saturated linear order of size $\mu^{<\mu}$. In fact, if
$\tau_0$ is any linear order of size $\mu^{<\mu}$, there is a $\mu$-saturated linear
order $\tau$ of size $\mu^{<\mu}$ such that $\tau_0 \leqslant \tau$. If $\mu^{<\mu} = \mu$, there is thus a
$\mu$-saturated linear order of size $\mu$. Any two such linear orders are isomorphic. Similarly, if $\tau$ is
$\mu$-saturated and $\varphi$ is a linear order of size $\mu$, then $\varphi \leqslant \tau$. In fact, the following holds.

\begin{lemma} \label{sum_embedding_lemma}
  Suppose $\tau$ is a $\mu$-saturated order type and $\varphi$ is an order type of size $\mu$. Then there
  are $\mu$-saturated order types $\soast{\tau_a}{a \in \varphi}$ such that $\sum_{a \in \varphi} \tau_a \leqslant \tau$.
\end{lemma}

\begin{proof}
  Let $P$ be an order of type $\varphi$ and let $\soast{a_\alpha}{\alpha < \mu}$ be an enumeration of its elements.
  Furthermore, let $T$ be an order of type $\tau$.
  By recursion on $\alpha < \mu$, define 
  $\langle b^0_\alpha, b^1_\alpha \mid \alpha < \mu \rangle$ such that:
  \begin{itemize}
    \item{for all $\alpha < \mu$, $b^0_\alpha, b^1_\alpha \in T$ and $b^0_\alpha <_T b^1_\alpha$;}
    \item{for $\alpha, \beta < \mu$, if $a_\alpha <_P a_\beta$, then $b^1_\alpha <_T b^0_\beta$ and,
      if $a_\beta <_P a_\alpha$, then $b^1_\beta <_T b^0_\alpha$.}
  \end{itemize}
  The construction is straightforward using the fact that $\tau$ is $\mu$-saturated. Also, since $\tau$ is
  $\mu$-saturated, for each $\alpha < \mu$, the interval $\opin{b^0_\alpha}{b^1_\alpha}{T}$ is itself a
  $\mu$-saturated order. For all $\alpha < \mu$, let $T_{a_\alpha} = \opin{b^0_\alpha}{b^1_\alpha}{T}$, and
  let $\tau_{a_\alpha} = \otp(T_{a_\alpha})$. Then $\bigcup_{a \in P} T_a$
  is a suborder of $T$ of type $\sum_{a \in \varphi} \tau_a$.
\end{proof}

\begin{proposition} \label{weakly_scattered_union}
  Suppose $T$ is a $\mu$-saturated linear order, $\nu < \mu$, and $\soast{T_\alpha}{\alpha < \nu}$
  is a collection of weakly $\mu$-scattered suborders of $T$.
  Then $\bigcup_{\alpha < \nu} T_\alpha \ne T$.
\end{proposition}

\begin{proof}

  We define sequences $\langle A_\alpha \mid \alpha < \nu \rangle$ and $\langle B_\alpha \mid \alpha < \nu \rangle$
  such that the following hold:
  \begin{enumerate}
    \item{For all $\alpha < \nu$, $A_\alpha, B_\alpha \subseteq T_\alpha$ and $\card{A_\alpha},
      \card{B_\alpha} < \mu$ (one or both of $A_\alpha$ and $B_\alpha$ may be empty).}
    \item{For all $\alpha < \beta < \nu$, $A_\alpha <_T A_\beta$ and $B_\beta <_T B_\alpha$.}
    \item{For all $\alpha, \alpha' < \nu$, $A_{\alpha} < B_{\alpha'}$.}
    \item{For all $\alpha < \nu$, $\opin{A_\alpha}{B_\alpha}{{T_\alpha}} = \emptyset$.}
  \end{enumerate}
  The construction is by recursion, as follows. Suppose $\beta < \nu$ and $\langle A_\alpha \mid \alpha < \beta \rangle$ and
  $\langle B_\alpha \mid \alpha < \beta \rangle$ have been defined. Let $C_\beta = \bigcup_{\alpha < \beta} A_\alpha$
  and $D_\beta = \bigcup_{\alpha < \beta} B_\alpha$. Note that $C_\beta <_T D_\beta$. By assumption, $T_\beta$
  is weakly $\mu$-scattered. In particular, $T_\beta \cap \opin{C_\beta}{D_\beta}{T}$ is not $\mu$-saturated.
  There are thus $A_\beta, B_\beta \subseteq T_\beta \cap \opin{C_\beta}{D_\beta}{T}$ such that $\card{A_\beta}, \card{B_\beta}
  < \mu$, $A_\beta <_T B_\beta$, and $\opin{A_\beta}{B_\beta}{T_\beta} = \emptyset$. $A_\beta$ and $B_\beta$ are
  then as desired.

  At the end of the construction, let $A = \bigcup_{\alpha < \nu} A_\alpha$ and $B = \bigcup_{\alpha < \nu} B_\alpha$.
  Then $\card{A}, \card{B} < \mu$ and $A <_T B$. Thus, since $T$ is $\mu$-saturated, there is $d \in T$ such that
  $A <_T \{d\} <_T B$. Suppose that, for some $\alpha < \nu$, $d \in T_\alpha$. Then $d \in \opin{A_\alpha}{B_\alpha}{{T_\alpha}}$,
  which contradicts requirement (4) in the construction. Thus, $\bigcup_{\alpha < \nu} T_\alpha \ne T$.
\end{proof}

This has two corollaries, the first of which is immediate.

\begin{corollary} \label{saturated_cor}
  Suppose $T$ is a $\mu$-saturated linear order, $\nu < \mu$, and $c:T \rightarrow \nu$. Then there is a $\mu$-saturated
  suborder $T' \subseteq T$ such that $c \restriction T'$ is constant.
\end{corollary}

\begin{corollary}
  Suppose $\kappa$ is a strongly inaccessible cardinal. Then there is a $\kappa$-scattered linear
  order that is not $\opair{\aleph_0}{\kappa}$-scattered.
\end{corollary}

\begin{proof}
  Let $\langle \mu_\zeta \mid \zeta < \kappa \rangle$ be an increasing sequence of regular
  cardinals, cofinal in $\kappa$. For each $\zeta < \kappa$, let $\varphi_\zeta$ be a $\mu_\zeta$-saturated
  linear order type of size $\mu_\zeta^{<\mu_\zeta}$. Since $\kappa$ is strongly inaccessible,
  $\card{\varphi_\zeta} < \kappa$ so, in particular, $\varphi_\zeta$ is $\kappa$-scattered. Let $\varphi =
  \sum_{\zeta < \kappa} \varphi_\zeta$. Then $\varphi$ is $\kappa$-scattered. We claim $\varphi$ is not
  $\opair{\aleph_0}{\kappa}$-scattered. Otherwise, there would be $\nu < \kappa$, an order $T$ of type
  $\varphi$, and a family
  $\soast{T_\xi}{\xi < \nu}$ of suborders of $T$ such that each $T_\xi$ is scattered and
  $T = \bigcup_{\xi < \nu} T_\xi$. Fix $\zeta < \kappa$ such that $\nu < \mu_\zeta$, and let
  $P \subseteq T$ be a $\mu_\zeta$-saturated linear order of size
  $\mu_\zeta^{< \mu_\zeta}$.
  Since $P$ is $\mu_\zeta$-saturated, by Lemma \ref{weakly_scattered_union},
  $P \neq \bigcup_{\xi < \nu} (T_\xi \cap P)$. This is a contradiction, so
  $\varphi$ is not $\opair{\aleph_0}{\kappa}$-scattered.
\end{proof}

On the other hand, it is straightforward to construct even a $\sigma$-scattered (i.e. $\opair{\aleph_0}{\aleph_1}$-scattered)
linear order of arbitrarily high density, so, for every regular, uncountable cardinal $\kappa$, there is
an $\opair{\aleph_0}{\kappa}$-scattered linear order that is not $\kappa$-scattered. Later results will imply that,
for every regular, uncountable $\kappa$, there is a weakly $\kappa$-scattered linear order that is not
$\opair{\aleph_0}{\kappa}$-scattered. Therefore, we will obtain the following corollary.

\begin{corollary} \label{inclusion_cor}
  Suppose $\kappa$ is a successor cardinal. Let $\mathcal{S}_\kappa$, $\mathcal{S}_{\aleph_0, \kappa}$, and $\mathcal{WS}_\kappa$
  be the classes of $\kappa$-scattered, $\opair{\aleph_0}{\kappa}$-scattered, and weakly $\kappa$-scattered linear order types,
  respectively. Then $\mathcal{S}_\kappa \subseteq \mathcal{S}_{\aleph_0, \kappa} \subseteq \mathcal{WS}_\kappa$,
  and neither of the inclusions is reversible.
\end{corollary}

We turn now to partition relations. Due to the restrictions identified by Theorem \ref{sierpinski_thm}, we will largely be interested
in partition relations of the form $\tau \arrows (\varphi, n)^2$,
where $n < \omega$. The following Lemma will be very useful.

\begin{lemma} \label{pairs_to_singletons_lemma}
  Suppose $\varphi$, $\tau$, and $\rho$ are linear order types, $n < \omega$, $\tau \arrows (\varphi, n)^2$,
  and, for all cardinals $\nu < \card{\varphi}$, $\rho \arrows (\tau)^1_\nu$. Then $\rho\varphi \arrows
  (\varphi, n+1)^2$.
\end{lemma}

\begin{proof}
  Let $\psi = \rho\varphi$, let $\opair{R}{<_R}$ be an ordered set of type $\rho$ and $\opair{P}{<_P}$ an ordered set of type $\varphi$.
  Let $f:[P \times R]^2 \rightarrow 2$. We will show that there is either a subset of
  $P \times R$ having order type $\varphi$ under the lexicographic order that is $0$-homogeneous for $f$ or a subset of size $n+1$ that is $1$-homogeneous
  for $f$. Enumerate the elements of $P$
  as $\seq{a_\zeta}{\zeta < \card{P}}$. We will attempt to construct a sequence $\seq{b_\zeta}{\zeta < \card{\varphi}}$ from $R$ such that
  $H \dfeq \soast{\opair{a_\zeta}{b_\zeta}}{\zeta < \card{P}}$ is $0$-homogeneous for $f$.
  Since the order type of $H$ under the lexicographic ordering is $\varphi$, a successful construction will establish the Lemma
  for this particular coloring..

  We proceed by recursion on $\zeta < \card{P}$. To start, let $b_0$ be an arbitrary element of $R$. Suppose $\zeta < \card{P}$
  and $b_\xi$ has been defined for $\xi < \zeta$ so that $\soast{\opair{a_\xi}{b_\xi}}{\xi < \zeta}$ is $0$-homogeneous for $f$.
  If there is $b \in R$ such that, for all $\xi < \zeta$,
  $f(\{\opair{a_\xi}{b_\xi}, \opair{a_\zeta}{b}\}) = 0$, then let $b_\zeta$ be such a $b$. Suppose there is no such $b$. Define a function
  $g: R \rightarrow \zeta$ by letting, for all $b \in R$, $g(b)$ be least such that $f(\{\opair{a_{g(b)}}{b_{g(b)}}, \opair{a_\zeta}{b}\}) = 1$.
  Since $\rho \arrows (\tau)^1_{\card{\zeta}}$, there is $\xi < \zeta$ and $B \subseteq R$ of order type $\tau$ such that $B$ is
  $\xi$-homogeneous for $g$. Consider $f \restriction [\{a_\zeta\} \times B]^2$. Since $\tau \arrows (\varphi, n)^2$, we either have a
  subset $B_0 \subseteq B$ of order type $\varphi$ such that $\{a_\zeta\} \times B_0$ is $0$-homogeneous for $f$ or a subset $B_1 \subseteq B$
  of size $n$ such that $\{a_\zeta\} \times B_1$ is $1$-homogeneous for $f$. In the former case, we have found a $0$-homogeneous
  subset of $P \times R$ of order type $\varphi$ and are thus finished. In the latter case, $\{\opair{a_\xi}{b_\xi}\} \cup (\{a_\zeta\} \times B_1)$ is
  a $1$-homogeneous subset of $P \times R$ of size $n+1$, and we are again done.

  Therefore, if our attempted construction of $H$ fails, it is necessarily because we have found either a $0$-homogeneous set of
  order type $\varphi$ or a $1$-homogeneous of size $n+1$. In any outcome, we have verified the Lemma.
\end{proof}

\label{sat_section}

We now easily have the following.

\begin{theorem}
 \label{theorem : pair-paradigma}
  Suppose $\kappa$ is an infinite, regular cardinal and $\varphi$ is a linear order type of size $\kappa$. Then, for all
  $n < \omega$ and all $\kappa$-saturated linear order types $\tau$,  $\tau \arrows (\varphi, n)^2$.
\end{theorem}

\begin{proof}
  We proceed by induction on $n$ simultaneously for all $\kappa$-saturated linear order types $\tau$. If $n \in \{0,1,2\}$, then
  trivially $\varphi \arrows (\varphi, n)^2$, so certainly $\tau \arrows (\varphi, n)^2$ for all $\kappa$-saturated
  $\tau$. Thus, suppose $2 \leqslant n < \omega$ and we have proven the theorem for $n$. Fix a $\kappa$-saturated type $\tau$
  . By Lemma \ref{sum_embedding_lemma}, we can find $\kappa$-saturated order types
  $\soast{\tau_a}{a \in \varphi}$ such that $\sum_{a \in \varphi} \tau_a \leqslant \tau$. By Corollary \ref{saturated_cor},
  the assumption that $\psi \arrows (\varphi, n)^2$ for all $\kappa$-saturated $\psi$, and the argument from
  Lemma \ref{pairs_to_singletons_lemma}, we have $\sum_{a \in \varphi} \tau_a \arrows (\varphi, n+1)^2$, so also
  $\tau \arrows (\varphi, n+1)^2$.
\end{proof}

\section{Weakly scattered linear orders} \label{weakly_scattered_section}
In this section, we establish positive partition relations for the class of weakly
$\kappa$-scattered linear orders under the assumption $\kappa^{<\kappa} = \kappa$. Our
aim is to prove the following theorem.

\begin{theorem} \label{weakly_scattered_thm}
  Suppose $\kappa^{<\kappa} = \kappa$ and $\varphi$ is a weakly $\kappa$-scattered linear
  order type of size at most $\kappa$. Then there is a weakly $\kappa$-scattered linear order type
  $\tau$ of size at most $\kappa$ such that, for all $n < \omega$, $\tau \arrows
  (\varphi, n)^2$.
\end{theorem}

Fix for the rest of the section a cardinal $\kappa$ such that $\kappa^{<\kappa} = \kappa$. The following lemma will be useful.

\begin{lemma} \label{unary_partition_lemma}
  Suppose $\varphi$ is a linear order type and $\nu$ is a cardinal. Then $\soafft{\nu}{\varphi} \arrows (\varphi)^1_\nu$.
\end{lemma}

\begin{proof}
  Let $P$ be an ordered set of type $\varphi$, let $F:\soafft{\nu}{P} \rightarrow \nu$, and suppose for sake of contradiction that there is no
  homogeneous set of order type $\varphi$ in the lexicographic order. By recursion on $\alpha \leqslant \nu$, we will construct
  $f_\alpha \in \soafft{\alpha}{P}$ such that, for all $\alpha < \beta \leqslant \nu$:
  \begin{enumerate}
    \item $f_\beta \restriction \alpha = f_\alpha$;
    \item for all $g \in \soafft{\nu}{P}$ such that $g \restriction (\alpha+1) = f_{\alpha+1}$, $F(g) \not= \alpha$.
  \end{enumerate}
  Suppose $\beta \leqslant \nu$ and we have constructed $f_\alpha$ for all $\alpha < \beta$. If $\beta$ is a limit
  ordinal, then easily $f_\beta = \bigcup_{\alpha < \beta} f_\alpha$ satisfies our requirements. Thus,
  suppose $\beta = \alpha + 1$. Suppose that, for all $a \in P$, there is $g \in \soafft{\nu}{P}$ such
  that $g \restriction \alpha = f_\alpha$, $g(\alpha) = a$, and $F(g) = \alpha$. Choose such a $g_a$ for each
  $a \in P$. Then $\soast{g_a}{a \in P}$ is an $\alpha$-monochromatic set for $F$ of order type
  $\varphi$, contradicting our assumptions. Thus, there is an $a \in P$ such that, for all $g \in \soafft{\nu}{P}$
  such that $g \restriction \alpha = f_\alpha$ and $g(\alpha) = a$, we have $F(g) \ne \alpha$. Choose such an $a$,
  and define $f_\beta$ by $f_\beta \restriction \alpha = f_\alpha$ and $f_\beta(\alpha) = a$.

  At the end of the recursion, we have constructed $f_\nu \in \soafft{\nu}{P}$. Let $\alpha = F(f_\nu)$. Then
  $f_\nu$ contradicts requirement (2) from the construction applied to $\alpha$ and $\nu$. Thus,
  $\soafft{\nu}{\varphi} \arrows (\varphi)^1_\nu$.
\end{proof}

\begin{lemma} \label{product_lemma}
  Suppose $P_0$ and $P_1$ are weakly $\kappa$-scattered linearly ordered sets. Then $\opair{P_0 \times P_1}{\ltl}$ is weakly $\kappa$-scattered.
\end{lemma}

\begin{proof}
  Suppose for sake of contradiction that $C \subseteq P_0 \times P_1$ is $\kappa$-saturated. Let
  \[
    T_0 = \soast{a \in P_0}{\mbox{there is }b \in P_1 \mbox{ such that }\opair{a}{b} \in C}.
  \]
  For $a \in T_0$,
  let $T_{1,a} = \soast{b \in P_1}{\opair{a}{b} \in C}$. The proof now splits into two cases.

  \textbf{Case 1: for all $a \in T_0$, $\card{T_{1,a}} = 1$.}
  In this case, we must have $\card{T_0} > 1$. Thus, choose $a_0 <_{P_0} a_1$, both in $T_0$.
  Since $P_0$ is weakly $\kappa$-scattered, $T_0 \cap \opin{a_0}{a_1}{P_0}$ is not $\kappa$-saturated.
  Thus, there are $A_0, A_1 \subseteq T_0 \cap \opin{a_0}{a_1}{P_0}$ such that $\card{A_0}, \card{A_1} < \kappa$,
  $A_0 <_{P_0} A_1$, and there is no $c \in T_0$ such that $A_0 <_{P_0} c <_{P_0} < A_1$. For
  each $a \in \{a_0, a_1\} \cup A_0 \cup A_1$, let $b_a$ be the unique $b \in P_1$ such that $\opair{a}{b} \in C$.
  For $i < 2$, let $B_i = \soast{\opair{a}{b_a}}{a \in \{a_i\} \cup A_i}$. Then $B_i \subseteq C$, $\card{B_i} < \kappa$,
  $B_0 \ltl B_1$, and there is no $c \in C$ such that $B_0 \ltl c \ltl B_1$,
  contradicting the assumption that $C$ is $\kappa$-saturated.

  \textbf{Case 2: there is $a \in T_0$ such that $\card{T_{1,a}} > 1$.}
  Choose such an $a \in T_0$, and fix $b_0 <_{P_1} b_1$, both in $T_{1,a}$.
  Since $P_1$ is weakly $\kappa$-scattered, $T_{1,a} \cap \opin{b_0}{b_1}{P_1}$ is not
  $\kappa$-saturated. Thus, there are $B_0, B_1 \subseteq T_{1,a} \cap \opin{b_0}{b_1}{P_1}$ such that
  $\card{B_0}, \card{B_1} < \kappa$ and $B_0 <_{P_1} B_1$, and there is no $c \in T_{1,a}$ such that
  $B_0 <_{P_1} c <_{P_1} B_1$. For $i < 2$, let $A_i = \soast{\opair{a}{b}}{b \in \{b_i\} \cup B_i}$.
  As in Case 1, $A_0$ and $A_1$ contradict the assumption that $C$ is $\kappa$-saturated.
\end{proof}

\begin{lemma} \label{exponent_lemma}
  Suppose that $\beta < \kappa$ and, for all $\alpha < \beta$, $\varphi_\alpha$ is weakly
  $\kappa$-scattered. Then $\prod_{\alpha < \beta} \varphi_\alpha$ is weakly $\kappa$-scattered.
\end{lemma}

\begin{proof}
  The proof is by induction on $\beta$. For $\beta \in \{0,1\}$, the Lemma is trivial. For the successor step
  of the induction,
  simply apply the induction hypothesis and Lemma \ref{product_lemma}. Thus, assume $\gamma < \kappa$ is a
  limit ordinal, $\varphi_\alpha$ is a weakly $\kappa$-scattered linear order type for all $\alpha < \gamma$ and,
  for all $\beta < \gamma$, $\prod_{\alpha < \beta} \varphi_\alpha$ is weakly $\kappa$-scattered.
  For all $\alpha < \gamma$, let $\opair{P_\alpha}{<_\alpha}$ be an ordered set of type $\varphi_\alpha$.
For notational
  simplicity, for $\beta \leqslant \gamma$, let $T_\beta = \bigtimes_{\alpha < \beta} P_\alpha$. Suppose for sake of
  contradiction that $R \subseteq T_\gamma$ is $\kappa$-saturated. For $f \in R$,
  let
  \begin{align*}
  D^+(f) \dfeq & \soast{\alpha < \gamma}{\exists g \in R ( \Delta(f,g) = \alpha \wedge f <_R g)};\\
  D^-(f) \dfeq & \soast{\alpha < \gamma}{\exists g \in R ( \Delta(f,g) = \alpha \wedge g <_R f)}.
  \end{align*}
  There are now three cases to consider.

  \textbf{Case 1: there is $f \in R$ such that $D^+(f)$ is unbounded in $\gamma$.}
  For $\alpha \in D^+(f)$, fix a $g_\alpha \in R$ such that $\Delta(f, g_\alpha) = \alpha$
  and $f <_R g$. Let $A = \{f\}$ and $B = \soast{g_\alpha}{\alpha \in D^+(f)}$.
  Then $A, B \subseteq R$, $\card{A}, \card{B} < \kappa$, $A <_R B$, and there is no $h \in R$
  such that $A <_R h <_R B$, contradicting the assumption that $R$ is $\kappa$-saturated.

  \textbf{Case 2: there is $f \in R$ such that $D^-(f)$ is unbounded in $\gamma$.}
  This is symmetric to Case 1.

  \textbf{Case 3: for all $f \in R$, there is $\beta < \gamma$ such that $D^+(f) \cup D^-(f) \subseteq \beta$.}
  For all $f \in R$, choose $\beta_f$ such that $D^+(f) \cup D^-(f) \subseteq \beta_f$. By
  Corollary \ref{saturated_cor}, there is $\beta < \gamma$ and a $\kappa$-saturated $S \subseteq R$
  such that, for all $f \in S$, $\beta_f = \beta$. Then, for all $f,g \in S$,
  $\Delta(f,g) < \beta$. In particular, $f \restriction \beta \ne g \restriction \beta$. Let
  $S_\beta = \soast{f \restriction \beta}{f \in S}$. Then $S_\beta$ is a $\kappa$-saturated
  suborder of $T_\beta$, contradicting the inductive hypothesis that $T_\beta$ is weakly
  $\kappa$-scattered.
\end{proof}

We are now ready to prove the main result of this section.

\begin{proof}[Proof of Theorem \ref{weakly_scattered_thm}]
  For $n < \omega$, we will find a weakly $\kappa$-scattered order of size $\kappa$, $\tau_n$,
  such that $\tau_n \arrows (\varphi, n)^2$. Then $\tau = \sum_{n < \omega} \tau_n$ will be
  as desired.

  We proceed by induction on $n < \omega$. For $n \in \{0,1,2\}$, we may simply set $\tau_n = \varphi$.
  Suppose $2 \leqslant n < \omega$ and $\tau_n$ has been found. Suppose first that $\kappa$ is a
  successor cardinal, say $\kappa = \nu^+$. In this case, let $\tau_{n+1} = \soafft{\nu}{\tau_n}$. If, on
  the other hand, $\kappa$ is a limit ordinal, let $\tau_{n+1} = \sum_{\nu < \kappa} \soafft{\nu}{\tau_n}$, where
  the sum is over all cardinals $\nu < \kappa$. In either case, by Lemma \ref{exponent_lemma},
  $\tau_{n+1}$ is weakly $\kappa$-saturated. Also, by Lemma \ref{unary_partition_lemma}, for all $\nu < \kappa$,
  $\tau_{n+1} \arrows (\tau_n)^1_\nu$. Therefore, by Lemma \ref{pairs_to_singletons_lemma},
  $\tau_{n+1} \arrows (\varphi, n+1)^2$.
\end{proof}

\begin{remark}
  Note that, for the case $\kappa = \nu^+$, for $2 \leqslant n < \omega$, the value for $\tau_n$ obtained
  in the above proof is $\soafft{\nu^{n-2}}{\varphi}$. In particular, we can take $\tau$ to be
  $\soafft{\nu^\omega}{\varphi}$.
\end{remark}

\section{A negative partition relation from a small unbounding number and the stick-principle}\label{cardinal_characteristics}

We define:
\begin{align}
\stick(\kappa) \dfeq \min\soast{\card{X}}{X \subseteq [\kappa^+]^\kappa \wedge \forall y \in [\kappa^+]^{\kappa^+}\exists x \in X \left(x \subseteq y\right)}.
\end{align}
\begin{remark}
$\stick(\kappa)$ is called $\stick(\kappa, \kappa^+)$ by Brendle in \cite{006B1}, general\iz ing $\stick$, which was introduced by Fuchino, Shelah and Soukup in \cite{997FSS0} and corresponds to our notion of $\stick(\aleph_0)$. Fuchino, Shelah and Soukup had in turn general\iz ed the stick principle $\stickp$ which can in hindsight be read as a shorthand for $\stick = \aleph_1$. The stick
principle was introduced by Broverman, Ginsburg, Kunen \& Tall in \cite[page 1311]{978BGKT0}.
$\stick(\kappa)$ as defined here should not be confused with $\stick_\lambda$ as defined in \cite{997FSS0} or $\stick_\delta$ as defined in \cite{018C0}.
\end{remark}

For an infinite cardinal $\kappa$, we let $\mathfrak{b}_\kappa$ and $\mathfrak{d}_\kappa$ denote, respectively, the unbounding number and the dominating number at $\kappa$ as they are defined by Cummings and Shelah in \cite{995CS0}. Here, $\mathfrak{b}$ and $\mathfrak{d}$ are shorthands for $\mathfrak{b}_{\aleph_0}$ and $\mathfrak{d}_{\aleph_0}$, respectively.

In \cite[Corollary 1, implicitly]{971EH0}, $\omega_1\omega \doesntarrow (\omega_1\omega, 3)^2$ is shown to follow from the Continuum Hypothesis.
In \cite{987T1}, Takahashi shows that the same consequence follows already from $\mathfrak{d} = \stick = \aleph_1$. In \cite{998L0}, the hypothesis $\mathfrak{d} = \stick = \aleph_1$ is weakened further and the result is generalised by Jean Larson, who shows that, for a regular cardinal $\kappa$, the hypothesis $\mathfrak{d}_\kappa = \kappa^+$ implies $\kappa^+\kappa \doesntarrow (\kappa^+\kappa, 3)^2$ . We provide here an improvement in a different direction.

\begin{theorem}
\label{theorem : stick + unbounding => partition}
Suppose that $\kappa$ is regular and $\lambda = \kappa^+ = \mathfrak{b}_\kappa = \stick(\kappa)$. Then $\lambda\kappa \doesntarrow (\lambda\kappa, 3)^2$.
\end{theorem}

\begin{proof}
 Assume towards a contradiction that the statement of the Theorem were wrong, \ie\ that $\mathfrak{b}_\kappa = \stick(\kappa) = \lambda$ and $\lambda\kappa \longrightarrow (\lambda\kappa, 3)^2$. Let $\stick(\kappa) = \lambda$ be witnessed by a family $D = \soast{d(\rho)}{\rho \in \lambda \setminus \kappa} \subseteq [\lambda]^\kappa$, and let $\mathfrak{b}_\kappa = \lambda$ be witnessed by a family $U = \soast{u_\rho}{\rho < \lambda}$ of increasing sequences of ordinals below $\kappa$ having the property that for all $\{\xi, \rho\}_< \in [\lambda]^2$ the set $\soast{\iota < \kappa}{u_\xi(\iota) \geqslant u_\rho(\iota)}$ has cardinality less than $\kappa$. We may also assume for notational convenience that $d(\rho) \subseteq \rho$ for all $\rho \in \lambda \setminus \kappa$. Finally, let $\seq{g_\gamma}{\kappa \leqslant \gamma < \lambda}$ be a sequence of bijections $g_\gamma : \kappa \longleftrightarrow \gamma$. We are now going to define a graph $E \subseteq [\lambda\kappa]^2$ by transfinite induction on the set of ordinals less than $\lambda$. We are going to have $E = \bigcup_{\rho < \lambda} E_\rho$. In step $\rho$ we will define $E_\rho \dfeq E \cap [\soast{\lambda \nu + \xi}{\xi < \rho \wedge \nu < \kappa}]^2$. The following is going to hold in every step $\rho$ of the induction:
\begin{align}
\label{formula : left upper corner & finiteness}
\forall \xi < \rho \forall \iota, \nu < \kappa\big(\card{\bigsoast{\vartheta < \xi}{\{\lambda \iota + \vartheta, \lambda \nu + \xi\} \in E}} < \kappa \wedge&\\
\forall \vartheta < \xi(\{\lambda \iota + \xi, \lambda \nu + \vartheta\} \in E \Rightarrow \iota < \nu)&\big)\notag
\end{align}
Let $E_0$ be the empty set. For positive limit ordinals $\rho < \lambda$, define $E_\rho \dfeq \bigcup_{\xi < \rho} E_\xi$, and for all ordinals $\rho < \lambda$ let $E_{\rho + 1} \dfeq E_\rho \cup D_\rho$, where $D_\rho$ is constructed as follows.

We inductively define a sequence $\seq{C_{\kappa\rho + \iota}}{\iota < \kappa}$ of sets $C_{\kappa\rho + \iota} \in [\rho]^{< \kappa}$ as follows: Let $\zeta$ be a ordinal less than $\kappa$ and suppose that $C_\xi$ has been defined for all $\xi < \kappa\rho + \zeta$. We now define $C_{\kappa\rho + \zeta}$ by letting
\begin{align}
\label{monster-one-line-definition}
C_{\kappa\rho + \zeta} \dfeq \Bigsoast{\min\Big(d\big(g_{g_\rho(\iota)}(\mu)\big) \Bigsetminus \bigcup_{\vartheta\in\bigcup_{\nu < \zeta} C_{\kappa\rho + \nu}} C_{\kappa\vartheta + \zeta}\Big)}{\iota < \zeta \wedge \mu < u_\rho(\zeta)}
\end{align}
Finally we set $D_\rho \dfeq \bigsoast{\{\lambda \iota + \rho, \lambda \nu + \xi\}}{\{\iota, \nu\} \subseteq \kappa \wedge \iota < \nu \wedge \xi \in C_{\kappa\rho + \nu}}$.
This completes the construction of $E$.

Now suppose that $T = \{\lambda \iota + \rho, \lambda \nu + \xi, \lambda \mu+ \vartheta\} \in [\lambda\kappa]^3$. If $[T]^2 \subseteq E$ then \wlg\ $\iota < \nu < \mu$ and $\vartheta < \xi < \rho$. Then $\xi \in C_{\kappa\rho + \nu}$ while $\vartheta \in C_{\kappa\xi + \mu} \cap C_{\kappa\rho + \mu}$, contradicting $\nu < \mu$ in conjunction with \eqref{monster-one-line-definition}. So by $\lambda\kappa \longrightarrow (\lambda\kappa, 3)^2$ there must be an $X \in [\lambda\kappa]^{\lambda\kappa}$ such that $[X]^2 \subseteq [\lambda\kappa]^2 \setminus E$. Let $A \dfeq \bigsoast{\iota < \kappa}{\card{\soast{\xi < \lambda}{\lambda \iota + \xi \in X}} = \lambda}$ and $\nu \dfeq \min(A)$. Furthermore, for all $\iota \in A$, let $\gamma_\iota \in \lambda \setminus \kappa$ be such that $\lambda \iota + \vartheta \in X$ for all $\vartheta \in d(\gamma_\iota)$. Set $\xi \dfeq \sup_{\iota < \kappa} \gamma_\iota$ and define a sequence of natural numbers $\seq{\sigma_\iota}{\iota < \kappa}$ by setting
$\sigma_\iota \dfeq \min(A \setminus \iota)$. We are going to define a function $f$ from $\kappa$ into itself:
\begin{align}
f : \kappa & \longrightarrow \kappa,\\
& \iota \longmapsto \max\big(\sigma_\iota, g_\xi^{-1}(\gamma_{\sigma_\iota})\big).\notag
\end{align}
Now let $\rho \in \lambda \setminus \xi$ be such that $\lambda \nu + \rho \in X$ and $u_\rho$ is unbounded over $f$. Set $\iota \dfeq g_\rho^{-1}(\xi)$ and let
$\tau \in \kappa \setminus \big(1 + \max(\iota, \nu)\big)$ be such that $u_\rho(\tau) > f(\tau)$. Let $\zeta = \sigma_\tau$. By definition of $f$, we have $\zeta < u_\rho(\tau)$.
Now consider \eqref{monster-one-line-definition} with $\mu \dfeq g_\xi^{-1}(\gamma_{\zeta})$. We have $\iota < \tau \leqslant \zeta = \sigma_\tau \in A$, and hence
\begin{align}
\mu = g_\xi^{-1}(\gamma_\zeta) \leqslant \max\big(\zeta, g_\xi^{-1}(\gamma_\zeta)\big) = \max\big(\sigma_\tau, g_\xi^{-1}(\gamma_{\sigma_\tau})\big) = f(\tau) < u_\rho(\tau) < u_\rho(\zeta).
\end{align}
Finally we have $g_\rho(\iota) = \xi$ and $A \ni \nu < \tau \leqslant \zeta \in A$. Then there is a $\vartheta \in d\big(g_{g_\rho(\iota)}(\mu)\big) = d(\gamma_\zeta)$ with $\vartheta \in C_{\kappa\rho + \zeta}$ so we get that $\{\lambda \nu + \rho, \lambda \zeta + \vartheta\} \in D_\rho \subseteq E_{\rho + 1} \subseteq E$, contradicting
the assumption that $[X]^2 \subseteq [\lambda \kappa]^2 \setminus E$.
\end{proof}

Observe that if $n$ is a natural number, $\iota$ and $\xi$ are ordinals, $\mu$ is a cardinal, $\alpha \in \mu^+ \setminus \mu$ and $E$ is an $(n + 1)$-hypergraph on $\iota\mu$ which witnesses $\iota\mu \doesntarrow (\iota\mu, \xi)^n$, then for any bijection $f : \alpha \leftrightarrow \mu$ witnessing $\alpha < \mu^+$ the $n$-hypergraph
\[
  \bigsoast{\{\iota\gamma_0 + \delta_0, \dots, \iota\gamma_n + \delta_n\}}{\soast{\iota\gamma_k + f(\delta_k)}{k \leqslant n} \in E}
\]
witnesses $\iota\alpha \doesntarrow (\iota\alpha, \xi)^n$. We therefore obtain the following corollary.

\begin{corollary}
\label{corollary : stick + unbounding => partition}
If $\kappa$ is regular and $\lambda = \kappa^+ = \mathfrak{b}_\kappa = \stick(\kappa)$, then $\alpha \doesntarrow (\lambda\kappa, 3)^2$ for all $\alpha < \lambda^2$.
\end{corollary}

In an upcoming paper of William Chen, Shimon Garti and the second author the same hypothesis is shown to imply $\lambda^2 \doesntarrow (\lambda\kappa, 4)^2$ as well.

Recall that $\clubsuit$, introduced in \cite{976O0}, implies $\stick = \aleph_1$. To see that this is the case, consider a sequence witessing the truth of $\clubsuit$. Its range then witnesses the truth of $\stick$. In \cite{997FSS0}, Fuchino, Shelah and Soukup show that $\clubsuit$ is consistent with $\cov(\meagre) = 2^{\aleph_0} = \aleph_2$.  Brendle \cite{017B0} points out that in their model the unbounding number is small as well. As $\cov(\meagre) \leqslant \mathfrak{d}$, we have
\begin{align*}
\ZFC + ``\aleph_1 = \stick = \mathfrak{b} < \mathfrak{d} = \aleph_2" \text{ is consistent,}
\end{align*}
which shows that our result does in fact provide information not given by Larson's result.

Brendle proves in \cite{006B1} that $\clubsuit$ is consistent with $\cov(\zero) = \aleph_2$. In the same paper in a footnote on page 45 he also gives a sketch of how to extract a proof of the consistency of $\clubsuit$ with $\add(\meagre) = \aleph_2$ from \cite{999DS0}. Before all this, Truss, in \cite{983T1}, shows limitations to these pursuits by proving---in effect---that if $\stick = \aleph_1$, then $\min\big(\cov(\meagre),\cov(\zero)\big) = \aleph_1$. 

\section{Scattered linear orders} \label{scattered_section}

In this section, we deal with $\kappa$-scattered linear orders for infinite, regular cardinals $\kappa$. We
first prove a general\iz ation of the Milner-Rado paradox \cite{965MR0} to the class of $\kappa$-scattered
linear orders of size $\kappa$. We call to your attention the classical statement of the paradox.

\begin{paradox}[{\cite[Theorem 5]{965MR0}}] \label{965MR0}
  Suppose $\kappa$ is an infinite cardinal and $\alpha < \kappa^+$. Then $\alpha$ can be written as
  $\bigcup_{n<\omega} X_n$ where, for all $n < \omega$, $\otp(X_n) < \kappa^{n+1}$.
\end{paradox}

A related statement can be found in a paper by Komj\'ath and Shelah:

\begin{lemma}[{\cite[Lemma 1]{003KS0}}] \label{lemma : komjath-shelah}
  Suppose $\varphi$ is a scattered order type and $S$ is a linear order of type $\varphi$.
  Then there is $f : S \rightarrow \omega$ such that, for all $n < \omega$, $f^{-1}(n)$ has no subset of order type
  $(\omega^* + \omega)^n$. Therefore, $\varphi \doesntarrow (\psi)^1_{\aleph_0}$ where
  $\psi = \sum_{n < \omega} (\omega^* + \omega)^n$.
\end{lemma}

Fix infinite cardinals $\kappa$ and $\mu$ such that $\cof(\mu) \geqslant \kappa$. Let $\mathcal{B}_{\kappa, \mu}$ be the class of linear order types $\varphi$ such that
either $\card{\varphi} < \kappa$ or $\varphi$ is a well-order or anti-well-order of size $\leqslant \mu$. The following
structure theorem for $\kappa$-scattered linear orders of size at most $\mu$ follows from Theorem \ref{hausdorff_structure_thm}.

\begin{theorem}[{\cite[Corollary 3.11]{012ABCDT0}}] \label{small_hausdorff}
  The class of $\kappa$-scattered linear orders of size at most $\mu$ is the smallest class of orders containing
  $\mathcal{B}_{\kappa, \mu}$ which is closed under lexicographic sums with index set in $\mathcal{B}_{\kappa, \mu}$.
\end{theorem}

The following Lemma has a straightforward proof and can be found, for example, in \cite[Lemma 5]{014L1}.

\begin{lemma} \label{order_type_prop}
  Suppose $\kappa$ is an infinite cardinal, $\nu < \cof(\kappa)$, and $m < \omega$. Suppose that, for each $\zeta < \nu$,
  $X_\zeta$ is a set of ordinals such that $\otp(X_\zeta) < \kappa^m$. Let $X = \bigcup_{\zeta < \nu} X_\zeta$.
  Then $\otp(X) < \kappa^m$.
\end{lemma}

A close relative is the following Lemma.

\begin{lemma}
\label{lemma : pigeonhole for order-types}
  Suppose $\alpha$ is an indecomposable ordinal, $\nu < \cof(\alpha)$, and $m < \omega$. Suppose that $X$ is an ordered set,
  $X = \bigcup_{\zeta < \nu} X_\zeta$, and, for each $\zeta < \nu$,
  $\otp(X_\zeta) \not\geqslant (\alpha\alpha^*)^m$.
  Then $\otp(X) \not\geqslant (\alpha\alpha^*)^m$.
\end{lemma}

\begin{proof}
Suppose towards a contradiction that the Lemma were false, and let $m$ be the largest natural number for which it holds. So $\otp(X) \geqslant (\alpha\alpha^*)^{1 + m}$. Let $Y$ be a set of order type $\alpha\alpha^*$ and let $f : \soafft{1 + m}{Y} \longrightarrow X$ be an embedding preserving the lexicographic ordering on $\soafft{1 + m}{Y}$. For $y \in Y$, define
\begin{align}
Z_y \dfeq \soast{x \in X}{\exists \vec{y} \in \soafft{m}{Y} (f(y, \vec{y}) \leqslant x) \wedge \exists \vec{y} \in \soafft{m}{Y} ( x \leqslant f(y, \vec{y}))},
\end{align}
and let $Z_{y, \zeta} \dfeq Z_y \cap X_\zeta$.

Note that for any $y, z\in Y$ with $y < z$ and any $x_0 \in Z_y$ and $x_1 \in Z_z$ we have $x_0 < x_1$.
Clearly, for every $y \in Y$ we have $Z_y = \bigcup_{\zeta < \nu} Z_{y, \zeta}$, and $f \upharpoonright (\{y\} \times \soafft{m}{Y)}$ yields an order-preserving embedding of $\soafft{m}{Y}$ into $Z_y$. Now we may use the inductive hypothesis and conclude that for every $y \in Y$ there is a $\zeta_y < \nu$ and an order-preserving embedding $f_y : \soafft{m}{Y} \longrightarrow Z_{y, \zeta_y}$.
Considering both the indecomposability and the cofinality of $\alpha$ twice, we conclude that there is an $A \in [Y]^{\alpha\alpha^*}$ and a $\zeta < \nu$ such that $\zeta_a = \zeta$ for all $a \in A$. Then the function
\begin{align}
g : A \times \soafft{m}{Y} & \longrightarrow X\\
\opair{a}{\vec{y}} & \longmapsto f_a(\vec{y}) \notag
\end{align}
is an order-preserving embedding of a set of order type $(\alpha\alpha^*)^{1 + m}$ into $X_\zeta$, a contradiction.
\end{proof}

An easy variation yields the fact that Lemma \ref{lemma : pigeonhole for order-types} remains true if
all instances of $\alpha \alpha^*$ are replaced by $\alpha^* \alpha$.
We are now ready to state and prove our general\iz ation of Theorem \ref{965MR0}.

\begin{theorem} \label{milner_rado_generalization}
 Let $\kappa, \mu$ be infinite regular cardinals such that $\kappa \leqslant \mu$, and suppose $\varphi$ is a
 $\opair{\kappa}{\max(\aleph_1, \kappa)}$-scattered linear order type of size at most $\mu$.
 Then every order $P$ of type $\varphi$ can be written as a union $P = \bigcup_{n < \omega} P_n$
 such that there is no $n < \omega$ for which $P_n$ has a suborder of type
 $\mu^n, {(\mu^n)}^*$, $(\kappa\kappa^*)^n$, or $(\kappa^*\kappa)^n$.
\end{theorem}

\begin{proof}
  We first prove the theorem assuming $\varphi$ is a $\kappa$-scattered linear order type of size at most $\mu$.
  We proceed by induction on the complexity of $\varphi$. If
  $\varphi \in \mathcal{B}_{\kappa, \mu}$, then the statement of the theorem is either trivial (if $\card{\varphi} < \kappa$)
  or is trivially implied by the Milner-Rado paradox (if $\varphi$ is a well-order or anti-well-order).
  It thus suffices to prove that, if $\varphi = \sum_{a \in \tau} \rho_a$, where $\tau$ and $\rho_a$ satisfy
  the statement of the theorem for all $a \in \tau$, then $\varphi$ satisfies the statement of the theorem.
  Fix such a $\varphi.$

  Let $T$ be an order of type $\tau$, and, for $a \in T$, let $R_a$ be an order of type
  $\rho_a$. Let $P = \soast{\opair{a}{b}}{a \in T \wedge b \in R_a}$ be equipped with the lexicographic order, and note
  that $\otp(P) = \varphi$. Fix $f:T \rightarrow \omega$ such that, for all $n < \omega$, $f^{-1}(n)$
  does not contain a suborder of type $\mu^n, {(\mu^n)}^*$, $(\kappa\kappa^*)^n$, or $(\kappa^*\kappa)^n$.
  Similarly, for all $a \in T$, fix $f_a: R_a \rightarrow \omega$
  such that, for all $n < \omega$, $f_a^{-1}(n)$ does not contain a suborder of type $\mu^n, {(\mu^n)}^*$, $(\kappa\kappa^*)^n$, or $(\kappa^*\kappa)^n$.
  Fix an injective function $\pi:\omega \times \omega \rightarrow \omega$ such that, for all $m,n < \omega$,
  $\pi(m,n) \geqslant m+n+1$. Define $g:P \rightarrow \omega$ by letting, for $a \in T$ and $b \in R_a$,
  $g(a,b) = \pi\big(f(a), f_a(b)\big)$. We claim that, for all $i < \omega$, $g^{-1}(i)$ does not contain a suborder of type
  $\mu^i, {(\mu^i)}^*$, $(\kappa\kappa^*)^i$, or $(\kappa^*\kappa)^i$.

  Suppose for sake of contradiction that there is an $i < \omega$ such that $g^{-1}(i)$
  contains a suborder of type $\mu^i$. There must be $m,n < \omega$ such that $\pi(m,n) = i$,
  as $g^{-1}(i)$ is empty for all other values of $i$. Let $P' \subseteq g^{-1}(i)$ have type
  $\mu^i$, and let $T' = \soast{a \in T}{\exists b \in R_a(\opair{a}{b} \in P')}$. $T'$ is
  then a well-ordered subset of $f^{-1}(m)$, so $\otp(T') < \mu^m$. For each $a \in T'$,
  let $R'_a = \soast{b \in R_a}{\opair{a}{b} \in P'}$. $R'_a$ is a well-ordered subset of $f_a^{-1}(n)$,
  so $\otp(R'_a) < \mu^n$. Moreover, $P' = \sum_{a \in T'} R'_a$, so
  $\otp(P') < (\mu^n) \cdot (\mu^m) = \mu^{m+n} < \mu^i$. Contradiction.

  The argument for ${(\mu^i)}^*$ is similar.

  Now suppose that there is an $i < \omega$ such that $g^{-1}(i)$ contains
  a suborder of type $(\kappa\kappa^*)^i$.
  Again there are natural numbers $m$ and $n$ such that $\pi(m, n) = i$. Let $P' \subseteq g^{-1}(i)$
  have type $(\kappa\kappa^*)^i$. As before, we may define
  $T' \dfeq \soast{a \in T}{\exists b \in R_a(\opair{a}{b} \in P'})$ and, for each $a \in T'$,
  $R'_a \dfeq \soast{b \in R_a}{\opair{a}{b} \in P'}$. We then have $T' \not\geqslant (\kappa\kappa^*)^m$ and $\otp(R'_a) \not\geqslant (\kappa\kappa^*)^n$.
  Now let $Q$ be a linear order of type $(\kappa\kappa^*)^{i-n}$ and, for all $q \in Q$, let $P_q \subseteq P'$ be a suborder of type
  $(\kappa\kappa^*)^n$ such that, for all $p < q$ in $Q$, $P_p < P_q$. Consider
  the function $f : Q \longrightarrow \Pot(T')$ given by
  $q \longmapsto \soast{a \in T'}{\exists b \in R_a(\opair{a}{b} \in P_q)}$. Clearly, for every $q \in Q$,
  the set $f(Q)$ is an interval in $T'$. Moreover for any $p, q, r \in Q$ with $p < q < r$ we have $f(p) \cap f(r) = \emptyset$, as
  $f(p) \cap f(r) \subseteq f(q)$ and $(\kappa\kappa^*)^n \not\leqslant \otp(R'_a)$ for every $a \in f(p) \cap f(q) \cap f(r)$.
  Consider a partition $\{Q_0, Q_1\}$ of $Q$ having the property that for both $i < 2$ and $p, q \in Q_i$ with $p < q$
  there is always an $r \in \opin{p}{q}{<_Q} \cap Q_{1 - i}$. Let $c$ be a choice function for $\pwim{f}{Q_0}$.
   Then $\otp(Q_0) = (\kappa\kappa^*)^{i - n} \geqslant (\kappa\kappa^*)^m$ and $c \circ f$ is an embedding
  of $Q_0$ into $T'$. Contradiction.

  The argument for $(\kappa^*\kappa)^i$ is similar.

  This finishes the proof for the special case of $\kappa$-scattered linear orders of cardinality at most $\mu$.
  Now we continue the proof to prove the general case stated above.

  Fix a $\opair{\kappa}{\max(\aleph_1, \kappa)}$-scattered linear order type $\varphi$ of size at most $\mu$.
  Let $P$ be an order of type $\varphi$, let $\nu < \max(\aleph_1, \kappa)$ be a cardinal, and, for
  each $\zeta < \nu$, let $P_\zeta$ be a $\kappa$-scattered suborder of $P$ such that $P = \bigcup_{\zeta < \nu} P_\zeta$.
  Without loss of generality, we may assume that the $P_\zeta$'s are pairwise disjoint. For each $\zeta < \nu$,
  let $f_\zeta:P_\zeta \rightarrow \omega$ be such that, for all $n < \omega$, $f_\zeta^{-1}(n)$ does not contain
  a suborder of type $\mu^n, {(\mu^n)}^*$, ${(\kappa\kappa^*)}^n$, or $(\kappa^*\kappa)^n$.

   Now we distinguish two non-exclusive cases which cover the issue at hand:

   First, assume that $\kappa$ is uncountable. For each $n < \omega$, let $f^{-1}(n) \dfeq \bigcup_{\zeta < \nu} f_\zeta^{-1}(n)$.
   Then $\nu < \kappa$ so, by Lemmas \ref{order_type_prop} and \ref{lemma : pigeonhole for order-types}, for all $n < \omega$, $f^{-1}(n)$ does not contain a suborder of type
   $\mu^n, (\mu^n)^*$, ${(\kappa\kappa^*)}^n$, or $(\kappa^*\kappa)^n$.

   Now assume that $\kappa < \aleph_2$. Then $\nu \leqslant \aleph_0$. Let $\iota : \omega \times \nu \longrightarrow \omega$
   be an injection such that for all $m < \omega$ and $n < \nu$ one has $\iota(m, n) \geqslant m$. For each $a \in P$, let $\zeta_a$
   be the unique $\zeta < \nu$ such that $a \in P_\zeta$, and let $f(a) \dfeq \iota\big(f_{\zeta_a}(a), \zeta_a\big)$.
   Clearly, for all $n < \omega$, $f^{-1}(n)$ does not contain a suborder of type $\mu^n, {(\mu^n)}^*$, ${(\kappa\kappa^*)}^n$, or $(\kappa^*\kappa)^n$.
\end{proof}

For $\kappa = \aleph_0$ and $\mu = \aleph_1$ this yields the following corollary.

\begin{corollary} \label{corollary : milner_rado_generalization}
 Suppose $\varphi$ is a $\sigma$-scattered linear order type of size at most $\aleph_1$.
 Then every order $P$ of type $\varphi$ can be written as a union $P = \bigcup_{n < \omega} P_n$
 such that there is no $n < \omega$ for which $P_n$ has a suborder of type
 $\omega_1^n, {(\omega_1^n)}^*$, $(\omega\omega^*)^n$, or $(\omega^*\omega)^n$.
\end{corollary}

We can now complete the proof of Corollary \ref{inclusion_cor}.

\begin{proposition}
  Suppose $\kappa$ is a regular, uncountable cardinal. Then there is a linear order of size $\kappa$ that is weakly
  $\kappa$-scattered but not $\opair{\kappa}{\kappa}$-scattered.
\end{proposition}

\begin{proof}
  Suppose first that $\kappa^{\aleph_0} = \kappa$. Let $\delta = \kappa^\omega$ (ordinal exponentiation), and let $\varphi = {^\omega}\delta$.
  Since $\kappa^{\aleph_0} = \kappa$, $\card{\varphi} = \kappa$. By Lemma \ref{exponent_lemma}, $\varphi$ is weakly $\kappa$-scattered. By Lemma
  \ref{unary_partition_lemma}, $\varphi \arrows (\kappa^\omega)^1_{\aleph_0}$. Therefore, by
  Theorem \ref{milner_rado_generalization}, $\varphi$ is not $\opair{\kappa}{\kappa}$-scattered.

  Next, suppose that $\kappa^{\aleph_0} > \kappa$. Then there is a regular $\mu < \kappa$ such that
  $\mu^{\aleph_0} > \kappa$. Let $P$ be a suborder of ${^\omega} \mu$ of size $\kappa$. By
  Lemma \ref{exponent_lemma}, $P$ is weakly $\kappa$-scattered. Also, ${^\omega} \mu$ has a
  dense suborder of size $\mu$, namely the set of $f \in {^\omega} \mu$ such that $f(n) = 0$ for
  all but finitely many $n < \omega$. Thus, $P$ does not contain a suborder of type $\kappa$ or $\kappa^*$.
  Suppose for sake of contradiction that $P = \bigcup_{\zeta < \nu} P_\zeta$, where $\nu < \kappa$
  and $P_\zeta$ is $\kappa$-scattered for all $\zeta < \nu$. Since $\card{P} = \kappa$, there is
  $\zeta < \nu$ such that $\card{P_\zeta} = \kappa$. But Theorem \ref{small_hausdorff} implies that
  every $\kappa$-scattered linear order of size $\kappa$ contains either a suborder of type $\kappa$ or a suborder of
  type $\kappa^*$, which is a contradiction.
\end{proof}

We next use Theorem \ref{milner_rado_generalization} to give a negative partition relation for $\opair{\aleph_1}{\aleph_1}$-scattered linear order types
of size at most $\aleph_1$, assuming $\omega_1\omega \doesntarrow (\omega_1\omega, 3)^2$.
\begin{theorem} \label{theorem : omega_1_negative_scattered}
  Assume that $\omega_1\omega \doesntarrow (\omega_1\omega, 3)^2$, and let $\tau$ be an
  $\opair{\aleph_1}{\aleph_1}$-scattered linear order type of size at most $\aleph_1$. Then
  $\tau \doesntarrow (\omega_1^{\omega} \disj {(\omega_1^{\omega})}^*, 3)^2$.
\end{theorem}

\begin{proof}
  Let $T$ be a linear order of type $\tau$, and use Theorem \ref{milner_rado_generalization} to write $T$ as a union $T = \bigcup_{n < \omega} T_n$
  such that, for all $n < \omega$, $\omega_1^\omega, (\omega_1^\omega)^* \not\leqslant \otp(T_n)$.
  Then use $\omega_1\omega \doesntarrow (\omega_1\omega, 3)^2$ to find $E \subseteq [T]^2$ such that $[X]^2 \not\subseteq E$ for all $X \in [T]^3$
  and $[Y]^2 \not\subseteq [T]^2 \setminus E$ for all $Y \subseteq \tau$ such that
  $\card{\soast{n < \omega}{\card{Y \cap T_n} = \aleph_1}} = \aleph_0$. It is easily verified that the function $f:[T]^2 \rightarrow 2$
  defined by $f(a,b) = 0$ iff $\{a,b\} \not\in E$ witnesses the negative partition relation.
\end{proof}



Since $\omega_1^\omega$ and $(\omega_1^\omega)^*$ are themselves scattered, and therefore
$\aleph_1$-scattered, order types of size $\aleph_1$, Theorem \ref{theorem : omega_1_negative_scattered},
when compared with Theorem \ref{weakly_scattered_thm}, provides an instance in which the classes of
$\aleph_1$-scattered or $\opair{\aleph_1}{\aleph_1}$-scattered linear orders behave quite differently from the class of weakly
$\aleph_1$-scattered linear orders.
By examining the proof of Theorem \ref{theorem : omega_1_negative_scattered}, it is evident that $\aleph_1$ can be
replaced by an arbitrary regular, uncountable $\kappa$ provided that $\kappa \omega \doesntarrow
(\kappa \omega, 3)^2$. This hypothesis does not always hold. In \cite{987SS0}, Shelah and Stanley
prove that, if $\kappa > \omega$ is regular and, for all $\mu < \kappa$, $\mu^{\aleph_0} < \kappa$,
then $\kappa \omega \arrows (\kappa \omega, n)^2$ for all $n < \omega$. However, in the same paper, they
show that, if $\kappa > \aleph_1$ is regular, then there is a ccc forcing extension in which
$\kappa \omega \doesntarrow (\kappa \omega, 3)^2$. We thus easily obtain the following corollary.

\begin{corollary}
  Suppose that $\kappa > \aleph_1$ is regular. Then there is a ccc forcing extension in which
  there is a $\kappa$-scattered linear order type $\varphi$ of size $\kappa$ such that, for all $\opair{\kappa}{\kappa}$-scattered
  linear order types $\tau$ of size $\kappa$, $\tau \doesntarrow (\varphi, 3)^2$.
\end{corollary}

We are also able to prove a slightly recondite theorem which, assuming $\omega_1\omega \doesntarrow (\omega_1\omega, 3)^2$, provides a stark contrast to
Theorem \ref{theorem : pair-paradigma} for $\kappa = \aleph_0$. It is strongly inspired by \cite[\S2, Corollary 2]{971EH0} and
features six different order types $\varphi$ of size $\aleph_1$ such that either $\varphi$ or $\varphi^*$ has one of the
following three characteristics of smallness:
\begin{enumerate}
\item It is well-ordered.
\item \label{careful} It is both scattered and the product of $\omega_1$ with a countable order-type.
\item Every proper initial segment is countable.
\end{enumerate}

In \eqref{careful}, note that the multiplication of order types fails to be commutative in general. Also note that the sole
order types satisfying any two of the above conditions are ordinals smaller than $\omega_1^2$.

\begin{theorem} \label{theorem : second omega_1_negative_scattered}
  Assume $\omega_1\omega \doesntarrow (\omega_1\omega, 3)^2$,
  let $\tau$ be a $\sigma$-scattered linear order type of size at most $\aleph_1$, and let $\rho$
  be an order type such that ${(\omega\omega^*)}^n \leqslant \rho$ for all natural numbers $n$. Then
  \begin{align*}
  \tau \doesntarrow \big(\omega_1^{\omega} \disj {(\omega_1^{\omega})}^* \disj \omega_1\rho \disj
  \omega_1^*\rho \disj \rho\omega_1 \disj \rho\omega_1^*, 3\big)^2.
  \end{align*}
\end{theorem}

\begin{proof}
  The first part of the proof follows the line of the proof of Theorem \ref{theorem : omega_1_negative_scattered}.

  Let $T$ be an order of type $\tau$, and use Corollary \ref{corollary : milner_rado_generalization} to write $T$ as a union $T = \bigcup_{n < \omega} T_n$
  such that, for all $n < \omega$, $T_n$ does not contain a suborder of type $\omega_1^\omega, {(\omega_1^\omega)}^*$, or $(\omega\omega^*)^n$.
  Use $\omega_1\omega \doesntarrow (\omega_1\omega, 3)^2$ to find $E \subseteq [T]^2$ such that $[X]^2 \not\subseteq E$ for all $X \in [T]^3$
  and $[Y]^2 \not\subseteq [T]^2 \setminus E$ for all $Y \subseteq T$ such that $\card{\soast{n < \omega}{
  \card{Y \cap T_n} = \aleph_1}} = \aleph_0$. Define $c:[T]^2 \rightarrow 2$ by $c(a,b) = 0$ iff $\{a,b\} \notin E$.
  We show that $c$ witnesses the negative partition relation, \ie\ that, for every suborder $Y$ of $T$ whose order type
  appears in the first coordinate of the partition relation, $[Y]^2 \cap E \neq \emptyset$.
  For $Y \in [T]^{\omega_1^\omega} \cup [T]^{{(\omega_1^\omega)}^*}$, this follows
  as in the proof of Theorem \ref{theorem : omega_1_negative_scattered}.

  Now let $\rho$ be an order-type such that ${(\omega\omega^*)}^n \leqslant \rho$ for all natural numbers $n$ and let $R$ be an order of type $\rho$.
  First assume that $Y \in [T]^{\omega_1\rho}$. Let
  $f : \opair{R \times \omega_1}{\ltl} \longleftrightarrow Y$ be order-preserving. By the pigeonhole principle we may assume \wlg\ that for all $r \in R$ there is an $n < \omega$ such that
  $\pwim{f}{\{r\} \times \omega_1} \subseteq T_n$. Since, for all $n < \omega$, ${(\omega\omega^*)}^n \leqslant \rho$ and ${(\omega\omega^*)}^n \not\leqslant \otp(\bigcup_{k < n} T_k)$,
  it follows that for all natural numbers $n$ there is an $r \in R$ and a $k \in \opin{n}{\omega}{<}$ such that
  $\pwim{f}{\{r\} \times \omega_1} \subseteq T_k$. This easily implies that $[Y]^2 \cap E \neq \emptyset$.

  Next, assume that $Y \in[T]^{\rho\omega_1}$.
  Let
  $f : \opair{\omega_1 \times R}{\ltl} \longleftrightarrow Y$ be order-preserving. We distinguish two cases:

  First assume that there are $\alpha < \omega_1$ and $n < \omega$ such that
  ${(\omega\omega^*)}^n \leqslant \otp\soast{r \in R}{f(\alpha, r) \in \bigcup_{k < n} T_k}$.
  But then ${(\omega\omega^*)}^n \leqslant \otp(\bigcup_{k < n} T_k)$ and so by Lemma \ref{lemma : pigeonhole for order-types}
  there is a $k < n$ with ${(\omega\omega^*)}^n \leqslant \otp(T_k)$, a contradiction.

  Next, assume that for all $\alpha < \omega_1$ and $n < \omega$ we have
  \[
    {(\omega\omega^*)}^n \not \leqslant \otp\soast{r \in R}{f(\alpha, r) \in \bigcup_{k < n} T_k}.
  \]
  This means that, in particular, $\pwim{f}{\{\alpha\} \times R} \not \subseteq \bigcup_{k < n} T_k$.
  Using the pigeonhole principle, let $X_0 \in [\omega_1]^{\omega_1}$, $r_0 \in R$, and $n_0 < \omega$ be such that $f(\alpha, r_0) \in T_{n_0}$
  for all $\alpha \in X_0$. Now, inductively, for every $k < \omega$, choose $X_{k + 1} \in [X_k]^{\omega_1}$, $r_{k + 1} \in R$, and
  $n_{k + 1} \in \opin{n_k}{\omega}{<}$ such that $f(\alpha, r_{k + 1}) \in T_{n_{k + 1}}$ for all $\alpha \in X_{k + 1}$.

  At the end of the inductive construction let
  $Z \dfeq \soast{f(\alpha, r_i)}{i < \omega \wedge \alpha \in X_i}$.
  Then $Z \subseteq Y$ and, clearly, $\card{\soast{n < \omega}{\card{Z \cap T_n} = \aleph_1}} = \aleph_0$, so $E \cap [Y]^2 \neq \emptyset$.

  The arguments for sets of order-type $\rho\omega_1^*$ and $\omega_1^*\rho$ are analogous, so this finishes the proof of the Theorem.
\end{proof}

Corollary \ref{corollary : stick + unbounding => partition} yields the following (\cf\ Remark \ref{remark : (alpha*alpha^*)^omega}):

\begin{corollary} \label{corollary : second omega_1_negative_scattered}
  Assume $\omega_1\omega \doesntarrow (\omega_1\omega, 3)^2$, and let $\tau$ be a $\sigma$-scattered linear order type of size at most $\aleph_1$. Then
  \begin{align}
  \tau \doesntarrow \big(\omega_1^{\omega} \disj {(\omega_1^{\omega})}^* \disj \omega_1(\omega\omega^*)^\omega \disj
  \omega_1^*(\omega\omega^*)^\omega \disj (\omega\omega^*)^\omega\omega_1 \disj (\omega\omega^*)^\omega\omega_1^*, 3\big)^2.
  \end{align}
\end{corollary}

\section{Generalizing a result of Komj\'ath and Shelah} \label{ks_section}

In this section, we general\iz e the following result of Komj\'ath and Shelah from \cite{003KS0} about partitioning
scattered linear orders.

\begin{theorem}
  If $\varphi$ is a scattered linear order and $\nu$ is a cardinal, then there is a scattered linear
  order $\psi$ such that $\psi \arrows [\varphi]^1_{\nu, \aleph_0}$.
\end{theorem}

We will prove the following general\iz ation.

\begin{theorem} \label{ks_generalization}
  Suppose $\kappa$ is a cardinal such that $\kappa^{<\kappa} = \kappa$, $\varphi$ is a $\kappa$-scattered
  linear order type, and $\nu$ is a cardinal. Then there is a $\kappa$-scattered linear order type $\psi$ such that
  $\psi \arrows [\varphi]^1_{\nu, \kappa}$.
\end{theorem}

The proof of Theorem \ref{ks_generalization} is a modification of Komj\'ath and Shelah's proof from \cite{003KS0}.
We will state without proof the main technical lemma from \cite{003KS0} but will provide the rest
of the details for completeness.

Suppose $\beta$ is an ordinal,
$\langle P_\alpha \mid \alpha < \beta \rangle$ is a sequence of linear orders and, for all
$\alpha < \beta$, we have specified a designated element $0_\alpha \in P_\alpha$. Then
$\bigoplus'_{\alpha < \beta} P_\alpha$ is a linear order whose underlying set is the set of
functions $f$ such that:
\begin{itemize}
  \item $\dom(f) = \beta$;
  \item for all $\alpha < \beta$, $f(\alpha) \in P_\alpha$;
  \item for all but finitely many $\alpha < \beta$, $f(\alpha) = 0_\alpha$.
\end{itemize}
If $f \in \bigoplus'_{\alpha < \beta} P_\alpha$, then $\mathrm{supp}(f)$ is the set of $\alpha < \beta$
such that $f(\alpha) \neq 0_\alpha$. $\bigoplus'_{\alpha < \beta} P_\alpha$ is ordered anti-lexicographically.
Namely, if $f,g \in \bigoplus'_{\alpha < \beta} P_\alpha$, let $\Delta'(f,g)$ denote the largest $\alpha < \beta$
such that $f(\alpha) \neq g(\alpha)$, and let $f <_{\bigoplus'_{\alpha < \beta} P_\alpha} g$ iff $f(\alpha) <_{P_\alpha} g(\alpha)$.
As usual, if there is $P$
such that $P_\alpha = P$ for all $\alpha < \beta$, we write $\bigoplus'_{\alpha < \beta} P$
in place of $\bigoplus'_{\alpha < \beta} P_\alpha$.

\begin{lemma} \label{antilex_lemma}
  Let $\beta$ be an ordinal and, for $\alpha < \beta$, let $P_\alpha$ be a linear order. Let
  $Q = \bigoplus'_{\alpha < \beta} P_\alpha$. Suppose $f,g,h \in Q$ are such that
  $f <_Q g <_Q h$. Then $\max(\Delta'(f,g), \Delta'(g,h)) \leqslant \Delta'(f,h)$.
\end{lemma}

\begin{proof}
  Let $\alpha = \Delta'(f,h)$. Suppose for sake of contradiction that $\gamma = \Delta'(f,g) > \alpha$.
  Then $f(\gamma) <_{P_{\gamma}} g(\gamma)$, and, since $f(\gamma) = h(\gamma)$, we also have $h(\gamma) <_{P_\gamma} g(\gamma)$.
  Since $g <_Q h$, we must have $\xi = \Delta'(g,h) > \gamma$. Then $g(\xi) <_{P_\xi} h(\xi) = f(\xi)$, contradicting
  the fact that $\xi > \gamma$ and $\Delta'(f,g) = \gamma$. Thus, $\Delta'(f,g) \leqslant \alpha$. $\Delta'(g,h) \leqslant \alpha$
  follows similarly.
\end{proof}

\begin{lemma} \label{second_antilex_lemma}
  Suppose $\beta$ is an ordinal, $\kappa$ is a regular cardinal, and, for all $\alpha < \beta$, $P_\alpha$
  is a $\kappa$-scattered linear order. Let $Q = \bigoplus'_{\alpha < \beta} P_\alpha$. Then $Q$ is
  $\kappa$-scattered.
\end{lemma}

\begin{proof}
  Suppose for sake of contradiction that $R$ is a $\kappa$-dense suborder of $Q$. Let $\alpha < \beta$
  be least such that, for some $f,g \in R$, $\Delta'(f,g) = \alpha$, and fix such $f,g \in R$.
  As $R$ is $\kappa$-dense, $\opin{f}{g}{R}$ itself must be $\kappa$-dense as a suborder of $R$. By the minimality of $\alpha$
  and Lemma \ref{antilex_lemma}, for all $h_0 <_R h_1$ in $\opin{f}{g}{R}$, we must have
  $\Delta'(h_0, h_1) = \alpha$ and hence $h_0(\alpha) <_{P_\alpha} h_1(\alpha)$. But this implies that
  $\{h(\alpha) \mid h \in R\}$ is a $\kappa$-dense suborder of $P_\alpha$, contradicting the assumption that
  $P_\alpha$ is $\kappa$-scattered.
\end{proof}

We now need some definitions and a lemma from \cite{003KS0}.

\begin{definition}
  Suppose $\alpha$ is an ordinal. $\mathrm{FS}(\alpha)$ is the set of all finite decreasing sequences
  from $\alpha$, i.e.\ sequences of the form $\vec{s} = \langle s_0, s_1, \ldots, s_{n-1} \rangle$ such
  that $\alpha > s_0 > s_1 > \ldots > s_{n-1}$. For such an $\vec{s}$ of length $n$, $\min(\vec{s})$ will
  denote $s_{n-1}$. An \emph{$\alpha$-tree} is a function $x:\mathrm{FS}(\alpha) \rightarrow \mathrm{On}$
  such that, for all $\vec{s} \in \mathrm{FS}(\alpha)$ and all $\gamma_0 < \gamma_1 < \min(\vec{s})$,
  we have $x(\vec{s}$ ${^\frown} \langle \gamma_0 \rangle) < x(\vec{s}$ ${^\frown} \langle \gamma_1 \rangle)
  < x(\vec{s})$.
\end{definition}

\begin{lemma} \label{ks_lemma}
  Suppose $\alpha$ is an ordinal and $\nu$ is a cardinal. Let $\mu = {\big(\card{\alpha}^{\nu^{\aleph_0}}\big)}^+$, and
  suppose that $F:\mathrm{FS}(\mu^+) \rightarrow \nu$. Then there is an $\alpha$-tree $x:\mathrm{FS}(\alpha)
  \rightarrow \mu^+$ and a function $c:\omega \rightarrow \nu$ such that, for every $n < \omega$ and every
  $\langle s_0, s_1, \ldots, s_n \rangle$ in $\mathrm{FS}(\alpha)$ of length $n+1$, we have
  $F(x(\langle s_0 \rangle), x(\langle s_0, s_1 \rangle), \ldots, x(\langle s_0, s_1, \ldots, s_n \rangle)) = c(n)$.
\end{lemma}

A proof of Lemma \ref{ks_lemma} can be found in \cite{003KS0}.

Fix now a cardinal $\kappa$ such that $\kappa^{<\kappa} = \kappa$. We next identify a class of $\kappa$-scattered
linear orders such that every $\kappa$-scattered linear order embeds into a member of the class. First, let
$\mathcal{T}$ be a set of linear orders such that, for every linear order type $\tau$ of size $<\kappa$,
there is a unique $T \in \mathcal{T}$ such that $\otp(T) = \tau$. Now, let $\langle \opair{T_\alpha}{0_\alpha}
\mid \alpha < \kappa \rangle$ enumerate all pairs $\opair{T}{a}$ such that $T \in \mathcal{T}$ and $a \in T$.
For the rest of the section, let $S$ denote $\bigoplus'_{\alpha < \kappa} T_\alpha$. If $\delta$ is an ordinal,
let $P_\delta = \bigoplus'_{\gamma < \delta} S$. We will sometimes think of elements of
$P_\delta$ as functions $f$ such that $\dom(f) = \delta \times \kappa$,
$f(\gamma, \alpha) \in T_\alpha$ for all $\opair{\gamma}{\alpha} \in \delta \times \kappa$, and, for all but
finitely many $\opair{\gamma}{\alpha} \in \delta \times \kappa$, $f(\gamma, \alpha) = 0_\alpha$. If $f,g$ are two
such functions, then $\Delta'(f,g)$ is the lexicographically largest $\opair{\gamma}{\alpha} \in \delta \times \kappa$
such that $f(\gamma, \alpha) \neq g(\gamma, \alpha)$, and $f <_{P_\delta} g$ iff $f(\gamma, \alpha) <_{T_\alpha}
g(\gamma, \alpha)$. By Lemma \ref{second_antilex_lemma}, for every ordinal $\delta$, $P_\delta$ is $\kappa$-scattered.
Let $\psi_\delta = \otp(P_\delta)$.

\begin{lemma}
  For every $\kappa$-scattered linear order $\varphi$, there is $\delta$ such that $\varphi \leqslant \psi_\delta$.
\end{lemma}

\begin{proof}
  We proceed by induction. If $\card{\varphi} < \kappa$, then there is $\alpha < \kappa$ such that
  $\otp(T_\alpha) = \varphi$. But then $\varphi \leqslant \otp(S) = \psi_1$. Suppose $\beta$ is
  an ordinal, $\varphi = \sum_{\xi < \beta} \varphi_\xi$ and, for each $\xi < \beta$, there is $\delta_\xi$
  such that $\varphi_\xi \leqslant \otp(P_{\delta_\xi})$. For $\xi < \beta$, let $Q_\xi$ be an order of type
  $\varphi_\xi$, and let $Q = \sum_{\xi < \beta} Q_\xi$, so $\otp(Q) = \varphi$.. Let $\delta' = \sup\soast{\delta_\xi}{\xi < \beta}$, and,
  for all $\xi < \beta$, fix an order embedding $F_\xi:Q_\xi \rightarrow P_{\delta'}$. Let $\alpha' <
  \kappa$ be such that $T_{\alpha'}$ is a 2-element linear order and $0_{\alpha'}$ is the smaller of the
  elements (call the other $1_{\alpha'}$). Let $\delta = \delta' + \beta$. Define $F:Q \rightarrow P_\delta$
  as follows. If $\xi < \beta$ and $a \in Q_\xi$, let $F(\xi, a) \in P_\delta$ be defined as
  follows.
  \[
    F(\xi, a)(\gamma, \alpha) =
    \begin{cases}
      F_\xi(a)(\gamma, \alpha) & \mbox{if } \opair{\gamma}{\alpha} \in \delta' \times \kappa \\
      1_{\alpha'} & \mbox{if } \opair{\gamma}{\alpha} = \opair{\delta' + \xi}{\alpha'} \\
      0_\alpha & \mbox{otherwise} \\
    \end{cases}
  \]
  It is easy to check that $F$ is an order embedding. The cases in which $\varphi$ is a
  lexicographic sum whose index set is anti-well-ordered or of size $<\kappa$ are similar.
\end{proof}

Theorem \ref{ks_generalization} will therefore follow from the following lemma.

\begin{lemma}
  For every ordinal $\delta$ and every cardinal $\nu$, there is $\mu$ such that
  $\psi_\mu \arrows [\psi_\delta]^1_{\nu, \kappa}$.
\end{lemma}

\begin{proof}
  Fix an ordinal $\delta$ and a cardinal $\nu$. Let $\mu = {\big(\card{\delta}^{\nu^\kappa}\big)}^{++}$.
  Let $H:P_\mu \rightarrow \nu$. We define a col\ou ring $F$ of $\mathrm{FS}(\mu)$ as
  follows. Suppose $\vec{s} = \langle s_0, \ldots, s_{n-1} \rangle \in \mathrm{FS}(\mu)$.
  Then $F(\vec{s})$ is the function defined on $\prod_{i < n} S$ defined as follows. If
  $\vec{a} = \langle a_0, \ldots, a_{n-1} \rangle \in \prod_{i < n} S$, then
  $F(\vec{s})(\vec{a}) = H(f)$, where $f \in P_\mu$ is such that
  $\mathrm{supp}(f) \subseteq \{s_0, \ldots, s_{n-1}\}$ and, for all $i < n$,
  $f(s_i) = a_i$. Since $\card{S} = \kappa$, $F$ is a col\ou ring with $\nu^\kappa$ col\ou rs.
  Therefore, by Lemma \ref{ks_lemma}, there is a $\delta$-tree $x:\mathrm{FS}(\delta)
  \rightarrow \mu$ and a function $c$ with domain $\omega$ such that, for every
  $n < \omega$ and every $\langle s_0, \ldots, s_{n-1} \rangle \in \mathrm{FS}(\delta)$,
  $F(x(\langle s_0 \rangle), \ldots, x(\langle s_0, \ldots, s_{n-1} \rangle)) = c(n)$.

  Fix such an $x$ and $c$. Define a function $h:P_\delta \rightarrow P_\mu$
  as follows. Suppose $f \in P_\delta$ and $\mathrm{supp}(f) = \langle s_0, \ldots,
  s_{n-1} \rangle$, listed in decreasing order. For each $i < n$, let $t_i =
  x(\langle s_0, \ldots, s_i \rangle)$. Let $h(f) \in P_\mu$ be such that
  $\mathrm{supp}(h(f)) = \soast{t_i}{i < n}$ and such that, for all $i < n$,
  $h(f)(t_i) = f(s_i)$.

  \begin{claim*}
    $h$ is order-preserving.
  \end{claim*}

  \begin{proof}
    Suppose $f,f' \in P_\delta$ with $f <_{P_\delta} f'$. Let $\mathrm{supp}(f) =
    \langle s_0, \ldots, s_{m-1} \rangle$ and $\mathrm{supp}(f') = \langle s'_0, \ldots,
    s'_{n-1} \rangle$, both listed in decreasing order. Let $i$ be least such that
    either $s_i \neq s'_i$ or $f(s_i) \neq f(s'_i)$.

    Suppose first that $s_i \neq s'_i$. Without loss of generality, assume
    $s_i < s'_i$ (the argument in the other case is symmetric). In particular, since
    $f <_{P_\delta} f'$, we have $0 <_S f'(s'_i)$. Then $x(\langle s_0, \ldots, s_i
    \rangle) < x(\langle s_0, \ldots, s'_i \rangle)$, so, by our definition of $h$,
    $\Delta'(h(f), h(f')) = x(\langle s_0, \ldots, s'_i \rangle) \eqdf t$ and
    $h(f)(t) = 0 <_S f'(s'_i) = h(f')(t)$, so $h(f) <_{P_\mu} h(f')$.

    If $s_i = s'_i$ and $f(s_i) <_S f(s'_i)$, let $x(\langle s_0, \ldots, s_i \rangle) = t$.
    Again, $\Delta'(h(f), h(f')) = t$ and $h(f)(t) = f(s_i) <_S f'(s_i) = h(f')(t)$, so
    $h(f) <_{P_\mu} h(f')$.
  \end{proof}

  Now let $P' = \pwim{h}{P_\delta}$. Then $P' \subseteq P_\mu$ and $\otp(P')
  = \psi_\delta$. We claim that $\card{\pwim{H}{P'}} \leqslant \kappa$.

  Recall that, for $n < \omega$, $c(n)$ is a function from $\prod_{i < n} S$
  to $\nu$. Let $A_n = \pwim{c}{\prod_{i < n} S}$, and let $A = \bigcup_{n < \omega} A_n$.
  Since $\card{S} = \kappa$, we have $\card{A} = \kappa$. We claim that
  $\pwim{H}{P'} \subseteq A$. To see this, let $g \in P'$. By definition, there is
  $f \in P_\delta$ such that $g = h(f)$. Let $\vec{s} = \langle s_0, \ldots, s_{n-1} \rangle =
  \mathrm{supp}(f)$ and, for $i < n$, $t_i = x(\langle s_0, \ldots, s_i \rangle)$, so
  $\mathrm{supp}(g) = \langle t_0, \ldots, t_{n-1} \rangle$. Let $\vec{a} = \langle a_0, \ldots,
  a_{n-1} \rangle$ be such that, for $i < n$, $f(s_i) = a_i = g(t_i)$. Retracing the definitions,
  we have:
  \[
    H(f) = F(\vec{t})(\vec{a}) = F(x(\langle s_0 \rangle), \ldots, x(\langle s_0, \ldots,
  s_{n-1} \rangle))(\vec{a}) = c(n)(\vec{a}) \in A_n \subseteq A.
  \]

  Thus, $\card{\pwim{H}{P'}} \leqslant \kappa$, completing the proof the theorem.
\end{proof}

\section{Questions}\label{section : questions}

At the end of this paper we are left with many open questions, some of which we would like to state explicitly.

Recall that Todorcevic showed in \cite[Chapter 2]{989T0} that $\mathfrak{b} = \aleph_1$ implies $\omega_1 \doesntarrow (\omega_1, \omega + 2)^2$. Together with Theorem \ref{theorem : stick + unbounding => partition} this suggests the following Question:

\begin{question}
Does $\mathfrak{b} = \aleph_1$ imply $\omega_1\omega \doesntarrow (\omega_1\omega, 3)^2$?
\end{question}

Both the statements $\forall n < \omega\left(\omega_1\omega \arrows (\omega_1\omega, n)^2\right)$ and $\forall n < \omega\left(\omega_1\omega^2 \arrows (\omega_1\omega^2, n)^2\right)$ were
shown to follow from $\MA_{\aleph_1}$ in \cite{989B0}. For the time being, we failed to answer the following Question:

\begin{question}
Does $\omega_1\omega \arrows (\omega_1\omega, 3)^2$ imply $\omega_1\omega^2 \arrows (\omega_1\omega^2, 3)^2$?
\end{question}

Regarding weakly scattered orders, the following questions seems to be central.

\begin{question} \label{strong_weakly_scattered_qstn}
  Is it consistent that there is a regular $\kappa$ such that for all weakly $\kappa$-scattered
  linear order types $\varphi$ of size $\kappa$ there is a weakly $\kappa$-scattered linear order type $\tau \geqslant \varphi$ of size
  $\kappa$ such that $\tau\arrows (\tau, 3)^2$?
\end{question}

\begin{question} \label{strong_scattered_qstn}
   Is it consistent that there is a regular $\kappa$ such that for all $\kappa$-scattered
  linear orders $\varphi$ of size $\kappa$ there is a $\kappa$-scattered linear order $\tau \geqslant \varphi$ of size
  $\kappa$ such that $\tau\arrows (\tau, 3)^2$?
\end{question}

The obvious candidate for $\kappa$ here is $\aleph_0$, but note that even the analogous question referring to ordinals
is unanswered. Another obvious question is whether the analogue of Theorem \ref{weakly_scattered_thm}
attained by replacing ``weakly $\kappa$-scattered" by ``$\kappa$-scattered" is consistently true.
This question is of interest both for successor cardinals and inaccessible cardinals.

\begin{question}
  Is it consistent that there is an uncountable, regular cardinal $\kappa$ such that, for every $\kappa$-scattered
  linear order $\varphi$ of size $\kappa$, there is a $\kappa$-scattered linear order $\tau$ of size
  $\kappa$ such that $\tau \arrows (\varphi, n)^2$ for all $n < \omega$?
\end{question}

A question which may be cumbersome to answer, is the following:

\begin{question}
Are all consistent negative partition relations for $\sigma$-scattered orders of cardinality $\aleph_1$ for two col\ou rs implied by the conclusion of Theorem \ref{theorem : second omega_1_negative_scattered}?
\end{question}

Finally, we ask whether Theorem \ref{ks_generalization} is optimal in terms of the
numbers of col\ou rs. In particular, we ask the following.

\begin{question}
Do the axioms of $\ZFC$ imply that there is a $\kappa$-scattered order type $\varphi$ such that, for every $\kappa$-scattered order type $\psi$, $\psi \doesntarrow [\varphi]^{1}_{\kappa, < \kappa}$?
\end{question}

\bibliography{thilo}
\bibliographystyle{thilo6}
\end{document}